\documentclass[12pt]{amsart}
\usepackage{epsfig}
\usepackage{subfigure}
\usepackage{amssymb,latexsym}
\usepackage{color,pstricks}
\usepackage{pgfpages}
\usepackage{lscape,pdflscape}
\usepackage{url}
\topskip  \numberwithin{equation}{section}
\newtheorem{theorem}{Theorem}[section]

\newtheorem{definition}[theorem]{Definition}

\newtheorem{remark}[theorem]{Remark}
\newtheorem{lemma}[theorem]{Lemma}

\newtheorem{example}[theorem]{Example}

\begin{document}

\def\NN{\mathbb N}
\def\L{\mathcal{L}}
\def\W{\mathcal{W}}
\def\N{\mathcal{N}}
\def\P{\mathcal{P}}
\def\({{\rm{(}}}
\def\){{\rm{)}}}
\def\c{{\rm{CN}}}
\def\p{{\mathbf p}}
\def\q{{\mathbf q}}
\def\e{{\mathbf e}}
\def\onev{{\mathbf 1}}
\def\ra{\rightarrow}

\def\+{\oplus}
\def\mn{\mbox{-}}
\def\newop#1{\expandafter\def\csname #1\endcsname{\mathop{\rm #1}\nolimits}}
\def\vs{\backslash_{\_} /}
\def\spic{\psdots*[dotstyle=*](0,0)
\rput(0.55,0){$\ldots$}
\psdots*[dotstyle=o](0.25,0)(0.9,0)
\psline(-0.1,-0.05)(-0.1,-0.1)(1.05,-0.1)(1.05,-0.05)
\rput(0.5,-0.25){$s$}}
\newcommand{\val}[2]{#1\begin{pspicture}(12pt,9pt)\psline[unit=4pt,fillcolor=black](0,2)(1,0)(2,0)(3,2)\end{pspicture}#2}
\newcommand{\arc}[2]{#1\begin{pspicture}(12pt,9pt)\pscurve[unit=4pt,fillcolor=black](0.2,0)(1.5,1.5)(2.8,0)\end{pspicture}#2}
\newcommand{\fl}[1]{\left\lfloor #1\right \rfloor}
\newcommand{\ceil}[1]{\left \lceil #1 \right\rceil}

\def\hl{\underline{h}}
\def\hu{\overline{h}}

\def\cb{\color{blue}}
\def\cg{\color[rgb]{0.1, 0.5, 0.2} }
\def\cp{\color[rgb]{0.5,0.2,0.5} }
\def\myred{\color{red}}
\newrgbcolor{purple}{0.7 0.2 0.7}
\newrgbcolor{orange}{1.0 0.5 0.0}
\newrgbcolor{mygreen}{0.1 0.5 0.2}

\makeatletter
\def\imod#1{\allowbreak\mkern10mu({\operator@font mod}\,\,#1)}
\makeatother

\pagenumbering{arabic}
\pagestyle{headings}
\def\sof{\hfill\rule{2mm}{2mm}}
\def\llim{\lim_{n\rightarrow\infty}}
\date{\today}
\title{Circular Nim games}
\maketitle

\begin{center}
{\bf Matthieu Dufour}\\
{\it Dept. of Mathematics, Universit\'e du Qu\'ebec \`a Montr\'eal\\
Montr\'eal,  Qu\'ebec H3C 3P8, Canada}\\
{\tt dufour.matthieu@uqam.ca}\\
\vskip 10pt
{\bf Silvia Heubach}\\
{\it Dept. of Mathematics, California State University Los
Angeles\\
Los Angeles, CA 90032, USA}\\
{\tt sheubac@calstatela.edu}\\

\end{center}

\section*{Abstract}
A circular Nim game is a two player impartial combinatorial game consisting of  $n$ stacks of tokens placed in a circle. A move consists of choosing $k$ consecutive stacks, and taking at least one token from one or more of the $k$ stacks.  The last player able to make a move wins.  We prove results on the structure of the losing positions for small $n$ and $k$ and pose some open questions for further investigations.

\noindent{\bf Keywords}: Combinatorial games, Nim, winning strategy

\noindent{\bf 2010 Mathematics Subject Classification}: 91A46, 91A05
\thispagestyle{empty}
\section{Introduction}\label{Introduction}

We consider circular Nim games, one of the many variations of the game of Nim. The game of Nim consists of several stacks of tokens.  Two players alternate taking one or more tokens from one of the stacks, and the player who cannot make a move loses. Nim is an example of an {\em impartial combinatorial game}, that is, all possible moves and positions in the game are known (there is no randomness), and  both players have the same moves available from a given position (unlike in Chess). Nim plays a central role among impartial games as any such game is equivalent to a Nim heap (see for example~\cite[Corollary 7.8]{AlbNowWol2007}). Nim has been completely analyzed and a winning strategy consists of removing tokens from one stack such that the {\em digital sum} of the heights of all stacks becomes zero.  The digital sum of two or more  integers in base $10$ is computed by first converting the integers into base $2$, then adding the base $2$ values without carry over, and then translating back into base $10$. We denote the digital addition operator by  $\+$. For example, $3 \+ 6\+ 14= 11$. Note that the digital sum $a \+ a = 0$ for all values of $a$.

The variation of Nim that we will consider is to arrange the $n$ stacks of tokens of a Nim game in a circle. In addition, we allow the  players to take at least one token from one or more of $k$ {\underline {consecutive}} stacks (as order now matters, unlike in the game of Nim). More specifically, if $p_j$ is the number of tokens in stack $j$, and $a_j$ is the number of tokens the player selects from stack $j$, then a {\em legal move} consists of picking stacks $i, i+1,\ldots,i+k-1$ (modulo $n$) for some $i=1,\ldots, n-1$, and then selecting $0 \le a_j \le p_j$ tokens from stack $j=i, i+1,\ldots,i+k-1$ with $\sum_{j=i}^{i+k-1}a_j\ge 1$. We denote this game by $\c(n,k)$. A {\em position} in a circular Nim game can be represented by a vector $\p=(p_1,p_2,\ldots,p_n)$ of non-negative entries indicating the heights of the stacks in order around the circle or any of its {\em symmetries}, namely the set of vectors $$\{(p_{\ell},p_{\ell+1},\ldots,p_n,p_1,\ldots,p_{\ell-1})\mid 1\le \ell\le n\}\cup \{(p_{\ell-1},p_{\ell-2},\ldots,p_1,p_n,\ldots,p_{\ell})\mid 1\le \ell\le n\},$$  where the indices are modulo $n$. The {\em final position} of $\c(n,k)$ is given by $(0,0,\ldots,0)$. Figure~\ref{8-3} visualizes a position in a $\c(8,3)$ game together with a possible choice of three stacks to play on.

\begin{figure}[htp]
\begin{center}
\epsfxsize=100.0pt  \epsffile{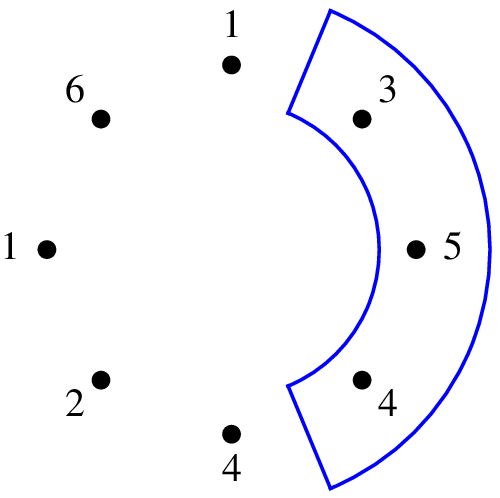}
\caption{The position $(1,3,5,4,4,2,1,6)$ in a $\c(8,3)$ game.}\label{8-3}
\end{center}
\end{figure}

We usually denote the current position in a game by $\p=(p_1,p_2,\ldots,p_n)$, and any position that can be reached by a legal move from $\p$ by $\p'=(p'_1,p_2',\ldots,p'_n)$. Such a position is called an {\em option} of $\p$, and we used the notation $\p \ra \p'$. We will also find it convenient in the proofs to use lowercase letters for the stack sizes to avoid the need for subscripts. In the same spirit of easy readability, we will refer to a stack by the number of its tokens, for example  as ``stack $a$'' or ``the $a$ stack'' instead of ``the third stack.'' If we need to make reference to a specific stack, we envision the first stack to be the one positioned at $12$ o'clock, and assume that the stacks are labeled in clockwise order. In Figure~\ref{8-3}, the third stack is a $5$ stack. In addition, we refer to the minimal value of $\p$ as $\min(\p)$, and the maximal value as $\max(\p)$, and to the vector $(1,1,\ldots,1)$ as $\onev$.

Usually, combinatorial games are studied from the standpoint of which player will win when playing from a given position. In this scenario, a position is either of type $\N$ or $\P$, where $\N$ indicates that the {\bf N}ext player to play from the current position has a winning strategy. The label $\P$ refers to the fact that the {\bf P}revious player, the one who made the move to the current position,  is the one to win (which means the player to play from the current position will lose no matter how s/he plays). We will take a slightly different (but equivalent) viewpoint, namely characterizing the position as either a winning or losing position for the player who goes next. Therefore, an $\N$ position is a winning position (as the next player wins), while a $\P$ position is a losing position. We will denote the set of winning and losing positions, respectively, as $\W$ and $\L$\footnote{In partizan games, $\L$ refers to the Left player.},  and characterize the set $\L$. For impartial games, the situation is remarkably simple.

\begin{theorem} \(see for example \cite[Theorem 2.11]{AlbNowWol2007}\) If $G$ is an impartial finite game, then for any position $\p$ of $G$, $\p \in \L$ or $\p \in \W$.
\end{theorem}


With this result,  determining either the set of winning or losing positions completely answers the question of whether the first or the second player has a winning strategy.  If we are discussing several games at the same time, then we will indicate the relevant game as a subscript for the set of losing positions, for example $\L_G$. Another well-known theorem will be crucial for the determination of the set of losing positions. 

\begin{theorem} \label{lose}\(see for example~\cite[Theorem 2.12]{AlbNowWol2007}\) Suppose the positions of a finite impartial game can be partitioned into mutually exclusive sets $A$ and $B$ with the properties:
\begin{itemize}
\item[ {\rm (I)}] every option of a position in $A$ is in $B$;
\item[{\rm (II)}] every position in $B$ has at least one option in $A$; and
\item[ {\rm (III)}] the final positions are in $A$.
\end{itemize}
Then $A=\L$ and $B=\W$.

\end{theorem}

Theorem~\ref{lose} tells us how to determine the set of losing positions. First we need to obtain a candidate set $S$ for the set of losing positions $\L$. Such a set $S$ may suggest itself when we examine patterns in the output of a computer program that determines the losing positions by recursively computing the Grundy function for each position. Once we have a candidate set, then we need to show that  any move from a position $\p \in S$ leads to a position $\p' \notin S$ (condition {\rm (I)}), and that for every position $\p \notin S$, there is a move that leads to a position $\p' \in S$ (condition  {\rm (II)}). Since $(0,0,\ldots,0)$ is the only final position, it is easy to see that condition {\rm (III)} is satisfied in all the proofs we give. Thus, showing that conditions {\rm (I)} and   {\rm (II)} are satisfied yields $S = \L$. Generally, it is relatively easy to show condition {\rm (I)}, while it may be quite difficult to show condition {\rm (II)}.

\section{The easy cases}\label{S:easy}

We first state a few easy general results.
\begin{theorem}\label{T:easy} \hfill
\begin{itemize}
\item[{\rm(1)}] The game  $\c(n,1)$ reduces to Nim, for which the set of losing positions is given by $\L=\{(p_1,p_2,\ldots,p_n)\mid p_1\+ p_2 \+ \cdots \+ p_n=0\}$.
\item[{\rm(2)}]  The game $\c(n,n)$ has a single losing position, namely $\L = \{(0,0,\ldots,0)\}$.
\item[{\rm(3)}]  The game $\c(n,n-1)$ has losing positions $\L = \{(a,a,\ldots,a)\mid a \ge 0\}$.
\end{itemize}
\end{theorem}

\begin{proof} {\rm(1)} This result can be found for example in the original analysis of Nim by Bouton~\cite{Bou1901}, in~\cite[Theorem 7.12]{AlbNowWol2007}, or the bible for combinatorial games~\cite{BerConGuy1}.  \\
{\rm(2)} In this game, the player playing from a position $\p \ne (0,0,\ldots,0)$ can always take all tokens from all stacks.\\
{\rm(3)} Let $S= \{(a,a,\ldots,a)\mid a \ge 0\}$.  Starting from a position $\p=(a,a,\ldots,a) \in S$,  at least one token has to be removed, so w.l.o.g., removal occurs at stack $1$, and the play is on stacks  $1, \ldots, n-1$. Thus, if the position after the play is $\p'=(p_1',p_2',\ldots, p_n')$, we have that $p_1'<a=p_n'$, and therefore, $\p' \notin S$, satisfying condition {\rm (I)}. On the other hand, from any position $\p \notin S$, we can reach a position in $S$ by finding the stack with the least number of tokens, and reducing the number of tokens in the $n-1$ other stacks to that minimal number of tokens, resulting in a position $\p'$ where all stacks have the same height, that is, $\p'\in S$. Thus, $S$ satisfies condition {\rm (II)}, which completes the proof.
\end{proof}

Note that Theorem~\ref{T:easy} completely covers the games $\c(n,k)$ for $n = 1, 2, 3$. For $n = 4$, the only game not covered is $\c(4,2)$.

\begin{theorem}\label{L4-2} For the game $\c(4,2)$, the set of losing positions is  $\L=\{(a,b,a,b)\mid a, b \ge 0\}$.
\end{theorem}

\begin{proof} Again we follow the directions of Theorem~\ref{lose} to determine the set $\L$. Let $S=\{(a,b,a,b)\mid a, b \ge 0\}$ and imagine the four stacks to be located at the corners of a square. For any position $\p=(p_1,p_2,p_3,p_4)=(a,b,a,b) \in S$, diagonally opposite stacks of the square have the same number of tokens. Any play on either one or two adjacent stacks affects only one stack of the diagonally opposite pair. Assuming  w.l.o.g. that the play is on stacks $1$ and $2$,  we have that $p'_1<p_1=p_3=p'_3$, and $p'_2\le p_2=p_4=p'_4$. Thus, $\p' \notin S$, and condition {\rm (I)} holds. On the other hand,  starting from any position $\p \notin S$, we determine the minimal value of each diagonal pair of stacks and reduce the stack with the larger number of tokens to the smaller value. This is always possible as any one stack is adjacent to both stacks of the other diagonal pair, so condition {\rm (II)} is satisfied. For example, for $\p=(3, 5, 4, 2) \notin \L$, reduce the second stack by three tokens and the third stack by one token to arrive at position $\p'=(3, 2, 3, 2)\in \L$.
\end{proof}

\section{Harder results }\label{main}

For $n=5$, the cases not covered by Theorem~\ref{T:easy} are $\c(5,2)$ and $\c(5,3)$. The result for $\c(5,2)$ was obtained by Dufour in his thesis~\cite{Duf}, and independently, by Ehrenborg and Steingr{\'{\i}}msson~\cite{EhrSte1996} as a special case of Nim played on a simplicial complex. {The results by Ehrenborg and Steingr{\'{\i}}msson  depend on the ability to explicitly obtain the circuits (see Definition~\ref{simp comp})   of the {\em cycle complex $C_{n,k}$}, which is possible only for small values of $n$ and $k$.} We will give elementary proofs of these results that do not rely on the framework of simplicial complexes.

\begin{theorem} \label{L5} \(see~\cite[Propositions 8.3 and 8.4]{EhrSte1996} and~\cite[Theorem 6.2.1]{Duf}\)
\begin{itemize}
\item[{\rm(1)}] The game  $\c(5,2)$ has losing positions $\L=\{(a^*,b,c,d,b)\mid a^*+b=c+d \mbox{ and } a^*=\max(\p)\}$.
\item[{\rm(2)}]  The game $\c(5,3)$ has losing positions $\L=\{(0,b,c,d,b)\mid b=c+d\}$.
\end{itemize}
\end{theorem}

Note that the conditions for $\c(5,2)$ force $b$ to be the minimal value, while the conditions for $\c(5,3)$ force $b$ to be maximal. Figure~\ref{n5} gives a visualization of the two sets of losing positions.

\begin{figure}[htp]
  \begin{center}
    \subfigure[$ \L_{\c(5,2)}$]{\label{fig:(5,2)}\includegraphics[scale=0.45]{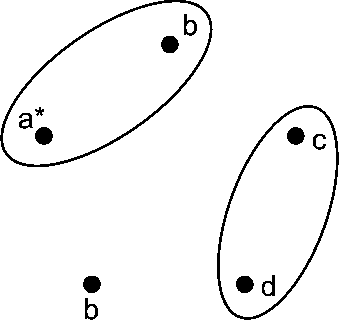}} \hspace{2cm}
    \subfigure[$\L_{\c(5,3)}$]{\label{fig:(5,3)}\includegraphics[scale=0.45]{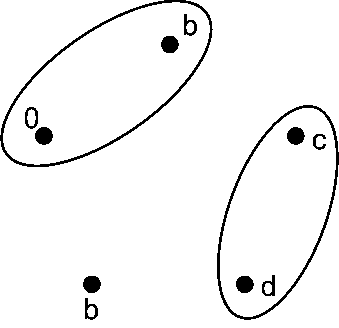}}   \end{center}
  \caption{Losing positions for $n=5$.}
  \label{n5}
\end{figure}

\begin{proof} 
{\rm(1)} Let $\p=(a^*,b,c,d,b)$ where $a^*+b=c+d$.  Play cannot be on a single stack, as it would destroy either the equality of the $b$ stacks, or the condition on the equality of the sums. Play on two stacks cannot include any of the $b$ stacks (as they cannot both be played), so the only choice is to play on $c$ and $d$, which results in $c'+d' < c+d=a^*+b$, violating the equality of sums. Thus, any move from $\p\in S$ will lead to a position $\p' \notin S$, and therefore, (I) holds. 

To show that we can move from any position $\p \notin S$ to a position $\p' \in S$, first note that $\p \in S\Leftrightarrow \p+m\cdot \onev \in S$ since the equality of the two sums and the equality of the $b$ stacks are not affected when a fixed amount is added or subtracted from each stack. Thus we may assume that $\min(\p)=0$.  We consider  two cases:
\begin{itemize}
\item[{\rm(i)}]  maximal and minimal value are adjacent; w.l.o.g.,  $\p=(0,w,x,y,z)$ and $w\ge x,  y, z$. If $w \ge z+y$, then $\p \ra (0,z+y,0,y,z) \in S$ is a legal move. For $w<z+y$, $\p \ra (0,w,0,w-z,z) \in S$ is a legal move. For example, $(0,6,4,3,2) \ra (0,5,0,3,2)$ and $(0,6,4,3,5) \ra (0,6,0,1,5)$;
\item[{\rm(ii)}] maximal and minimal values are separated by one stack; w.l.o.g., $\p=(0,x+y,w,z,y)$, and $\max(\p)\in\{w,z\}$.  If $z \ge x$, then  $\p\ra (0,x+y,0,x,y) \in S$ is a legal move. Otherwise $\p\ra (0,z+y,0,z,y) \in S$ is a legal move. For example, $(0,5,6,3,4)\ra(0,5,0,1,4)$, and $(0,5,6,1,3)\ra(0,4,0,1,3)$.
\end{itemize}
This completes the proof that $S=\L$ for $\c(5,2)$.\\
{\rm(2)} Now we look at the case $\c(5,3)$ and rewrite the structure of the losing positions, letting $S=\{(0,a+b,a,b,a+b)\}$. Now we are allowed to play on three stacks. If play involves either the $a$ or $b$ stack, then both $a+b$ stacks have to change, which would mean play on four stacks, which is not allowed. If  the play is on the other three stacks, then we have to reduce both $a+b$ stacks by the same amount, but their height no longer is the sum of the height of the $a$ and $b$ stacks, so condition (I) holds. To show the validity of condition (II), we let $\min(\p)=m$ and $\max(\p)=M$, and again consider the two cases where $\min(\p)$ and $\max(\p)$ are either adjacent or one stack apart.
\begin{itemize}
\item[{\rm(i)}] max($\p$) and min($\p$) are adjacent, w.l.o.g.,  $\p=(m,M,x,y,z)$. We display the different cases and examples of moves in a table, with  stacks that remain fixed underlined:\\
\begin{center}
\begin{tabular}{c|c|c}\hline
Case & $\p'$ & Example\\ \hline
$y-z \ge m$ & $(\underline{m},m+z,0,m+z,\underline{z})$ & $(3,9,5,7,4) \ra (3,7,0,7,4)$ \\ \hline
$0 \le y-z<m$ & $(y-z,y,0,\underline{y},\underline{z})$ & $(3,9,5,6,4) \ra (2,6,0,6,4)$ \\ \hline
$y-z <0 \wedge x > z-y$  &$(0,z,z-y,\underline{y},\underline{z})$ & $(3,6,4,3,5) \ra (0,5,2,3,5)$ \\ \hline
$y-z <0 \wedge x \le z-y$  &$(0,x+y,\underline{x},\underline{y},x+y)$ & $(3,6,1,3,5) \ra (0,4,1,3,4)$ \\ \hline
\end{tabular}
\end{center}

\vspace{0.2in}

\item[{\rm(ii)}] max($\p$) and min($\p$) are separated by one stack; w.l.o.g., $\p=(m,x,M,y,z)$.  If $x \ge z-m$, then  $\p\ra (m,z-m,z,0,z) \in S$ is a legal move. Otherwise $\p\ra (m,x,x+m,0,x+m) \in S$ is a legal move. For example, $(2,5,8,7,3)\ra(2,1,3,0,3)$ and $(2,3,8,7,3)\ra(2,3,5,0,5)$.
\end{itemize}
This completes the proof that $S=\L$ for $\c(5,3)$.
\end{proof}

Now we turn to results for $n=6$.  Figure~\ref{n6-3} visualizes the set of losing configurations.

\begin{theorem} \label{L6-3}
For the game $\c(6,3)$,  the set of losing positions is given by  $\L=\{(a,b,c,d,e,f)\mid a+b = d+e \text{ and } b+c = e+f\}$.\footnote{This result has also recently been discovered  independently and appeared  in {\cite[Example 23]{Hor2010}}. Once more, we provide an elementary proof that does not rely on the framework of simplicial complexes.}
\end{theorem}

\begin{figure}[htp]
\begin{center}
\epsfxsize=70.0pt  \epsffile{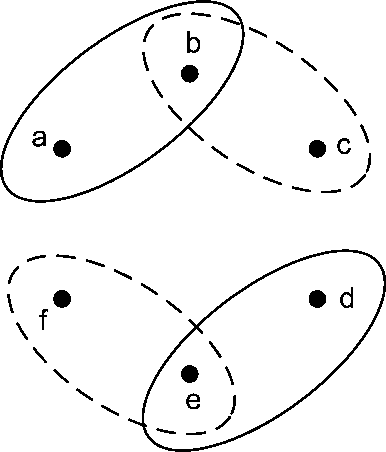} 
\caption{Losing positions for $\c(6,3)$.}\label{n6-3}
\end{center}
\end{figure}

\begin{remark} Note that for positions in the losing set given in Theorem~\ref{L6-3}, two pairs of opposite stacks have equal sums. However, having two sets of opposite pairs with the same sum also forces the third set of opposite pairs to have equal sums. Thus when proving results about the losing set,  we are done as soon as we have shown that any two sets of opposite pairs have the same sum. This will come in handy in the proof that follows. Alternatively, the symmetries indicate that it does not matter which two sets of opposite pairs have the same sum.
\end{remark}

\begin{proof}
Let $S=\{(a,b,c,d,e,f)\mid a+b = d+e \text{ and } b+c = e+f\}$. Suppose that $\p \in S$, and w.o.l.g, the move is made on the three consecutive stacks $a$, $b$ and $c$, so  $\p=(a,b,c,d,e,f)\ra \p'=(a', b', c', d,e, f)$. At least one token is removed, so w.l.o.g. suppose $a'<a$. Then $a'+b'<a+b =d+e$, so  $\p'\notin S$ and (I) holds. To show  condition (II), assume that $\p \notin S$ and observe that if there is a legal move $\p  \ra \p'$, then  there is a legal move from $\p+\ell\cdot \onev \ra \p'+\ell\cdot \onev$, for any positive integer value of $\ell$.  Therefore, we can assume w.l.o.g.  that $f=0$. Also, due to the circular symmetries, one can assume that $a+b \ge d+e$ (*). Three cases need to be considered: 
\begin{itemize}
\item[\rm{(i)}] $b > e$: Play is on stacks $b$ and $c$ and on either  stack $a$ or $d$, depending on which value is bigger; $\p \ra (\min(a,d),e,0,\min(a,d),e,0) \in S$ is a legal move. For example, $(5, 10, 8, 6, 9, 0) \ra (5, 9, 0, 5, 9, 0)$;

\item[\rm{(ii)}]  $b \le e \wedge c \ge e-b$:
We play on  stacks $a$, $b$, and $c$. Condition (*) guarantees that $a\ge d+e-b$, and thus $\p \ra (d+e-b,b,e-b,d,e,0)\in S$ is a legal move.  For example, $(10, 8, 8, 4, 9, 0) \ra (5, 8, 1, 4, 9, 0)$. (Note that if $a+b=d+e$ and $c=e-b$, then $\p \in S$, a contradiction.)

\item[\rm{(iii)}]  $b \le e \wedge c < e-b$:
 In this case, we play on stacks $e$, $f$, and  $a$. Since 
$a \ge d + e - b > d + c$, $\p \ra (c+d,b,c,d,b+c,0)\in S$ is a legal move. 
For example,  $(10, 8, 5, 2, 14, 0) \ra(7, 8, 5, 2, 13, 0)$.

\end{itemize}
In all cases, we can move from any $\p \notin S$ to $\p' \in S$, thus condition (II) holds and therefore $S=\L$.
\end{proof}

 We will discuss in Section~\ref{GenRes} why the proof given in \cite{Hor2010} does not extend to other cases.

\begin{remark} The proof of Theorem~\ref{L6-3} illustrates just one way of making a move from a position not in $\L$ to a position in $\L$. In general, this move is not unique. For example, for $\p=(a, b, c ,d, e, f) \notin \L$, $\p'=(a'+ \ell, b' - \ell, c' + \ell, d, e, f)\in \L$  for all the values of $\ell$ that preserve the legality of the move, that is, $a' + \ell \le a$, $ c' + \ell \le c$, and  $b' - \ell\ge 0$. As an illustration, from the position $\p=(10, 9, 5, 8, 4, 3)$, one can move to the positions $\p'= (5 + \ell, 7 - \ell, 0 + \ell, 8, 4, 3)\in \L$ for $\ell = 0 , 1, 2, \ldots, 5$.
\end{remark} 

We next present the result for the game $\c(6,4)$.  Figure~\ref{n6-4} visualizes the set of losing positions,  which are very similar to those in the game $\c(6,3)$, with additional properties involving a digital sum.

\begin{theorem} \label{L6-4}
For the game $\c(6,4)$,  the set of losing positions is given by  $$\L=\{(a,b,c,d,e,f)\mid a+b = d+e, b+c = e+f, a\+c\+e=0, \mbox{ and } a=\min(\p)\}.$$\end{theorem} 

\begin{figure}[htp]
\begin{center}
\epsfxsize=70.0pt  \epsffile{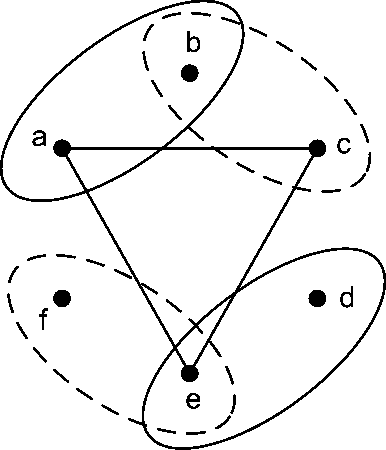} 
\caption{Losing positions for $\c(6,4)$.}\label{n6-4}
\end{center}
\end{figure}

As before, the third set of opposite pairs also has to have equal sums. In addition,  a losing position in which the minimum occurs simultaneously in each of the triples $(a, c, e)$ and $(b,d,f)$ reduces to a special case. 

\begin{lemma}\label{two minima} If the position $\p=(a,b,c,d,e,f) \in \L_{\c(6,4)}$ has its minimal value in each of the two triples $(a, c, e)$ and $(b,d,f)$, then $\p=(a,b,c,a,b,c)$.
\end{lemma}

\begin{proof} There are two cases to be considered: The minima are adjacent, or they are not adjacent. Assume that w.l.o.g. that the two adjacent minima occur at $a$ and $b$. Since $a=b$, we have $d=e=a$ (because of the minimality of $a$ and $b$), and consequently, due to the equality of the paired sums, $c=f$. For the second case assume the minima occur at $a$ and $d$. Then $e=b$ and $f=c$ because of the equality of paired sums.
\end{proof}

In addition, we make use of a well-known result about digital sums.

\begin{lemma}\label{digsumzero} For any set of positive integers $x_1$, $x_2, \ldots, x_n$ whose digital sum is not equal to zero, there exists an index $i$ and a value $x_i'$ such that $0 \le x_i' < x_i$ and
$x_1\+\cdots \+x_{i-1}\+x_i'\+x_{i+1}\+\cdots\+x_n=0.$
\end{lemma}
 

We are ready to prove Theorem~\ref{L6-4}.
 
\begin{proof}  Let $S=\{(a,b,c,d,e,f)\mid a+b = d+e, b+c = e+f, a\+c\+e=0 \}$ and let $\p \in S$. Note that we have not yet indicated where the minimum occurs, but we assume that it occurs at either $a$, $c$ or $e$. If play is on one, two, or three consecutive stacks, then any move from $\p \in S$ is to $\p' \notin S$ as in the game $\c(6,3)$. Therefore, play has to occur on four consecutive stacks, w.l.o.g., on stacks $a$ through $d$. We now attempt to make a move to another position in $S$. Since stacks $e$ and $f$ do not change, we cannot have a reduction in stacks $b$ and $c$ as the sums have to remain equal. Therefore, play is only on stacks $a$ and $d$, and these two stacks have to be reduced by the same amount, that is $\p'=(a-x,b,c,d-x,e,f)$ for some $x>0$. Let us refer to a triple whose stack heights have digital sum zero as a {\em digital triangle}. Since the digital triangle of $\p$ is $(a,c,e)$ and only stack $a$ is changed, the triangle $(a-x, c, e)$ is no longer digital, so  $(b, d-x,f)$ has to be the digital triangle of $\p'$. If the minimal value in the digital triangle of $\p$ is $a$, then $a-x$ is the only minimum in $\p'$ and it is not part of the digital triangle, so $\p'\notin S$. On the other hand, if the minimum of $\p$ occurs at either $c$ or $e$, then for $\p'$ to be in $S$, the minimum for $\p'$ has to occur in the digital triangle $(b, d-x, f)$. Since only the value of $d$ has changed in those stacks, then $d-x$ has to be the  minimal value of $\p'$. We need to consider two subcases:  the minimum occurs in both triangles of $\p'$, or the minimum of $\p'$ is unique. In the first subcase, Lemma~\ref{two minima} tells us that $\p'$ is of the form $(a,b,c,a,b,c)$; therefore, $\p'=(a-x,b,c,a-x,b,c)\notin S$ as $\p'$ does not have a digital triangle.  In the second case, we may assume w.l.o.g. that $\min(\p)=e$, and therefore,  $e < b$. Since the minimum of $\p'$ is unique, $d-x<a-x$, which implies that $d < a$. Combining the inequalities leads to $d+e<a+b$, so $\p' \notin S$. As there is no legal move from a position in $S$ to another position in $S$,  condition (I) is satisfied. 

Now we turn to the harder part, namely showing that from any position  $\p \notin S$, we can make a legal move to a position $\p' \in S$. Note that the condition to have equal sums for diagonally opposite pairs of stacks is equivalent to the condition $a-d=e-b=c-f$, that is, the differences between diagonally opposite stacks is the same for all such pairs.  To create a position $\p' \in S$ from a position $\p \notin S$ we proceed in two steps - first we create the digital triangle, and then we adjust the pairwise  differences of diagonally opposite pairs. To better visualize the relative sizes of stacks, we will label pairwise diagonally opposite values with the same letter, using lowercase for the smaller of the two and uppercase for the larger one. There are two different cases:
\begin{enumerate}
\item[1.] the pairwise minima are alternating with pairwise maxima (and thus form a triangle); or
\item[2.] the pairwise minima are all consecutive.
\end{enumerate}
To show that there is no third case, consider what happens when two of the pairwise minima are next to each other. The values of the third pair are adjacent to the two pairwise minima, and one of the two values has to be the pairwise minimum, making all the pairwise minima adjacent to each other. 

We now look at the two cases separately. Even though they have much in common, to combine them would create cumbersome notation). \\
Case 1: Let $\p=(A,b,C,a,B,c)$. By Lemma~\ref{digsumzero}, we can adjust one of the three pairwise minima to create a digital sum of zero (if not already digital). Assume that the value adjusted is $a$, and the new value is $\tilde{a} \le a \le A$. Compute the minimal pair difference $m =\min(A-\tilde{a}, B-b,C-c)$. In order to make all the pairwise differences equal to $m$, we need only adjust two of the pairwise maxima. If $m=A-\tilde{a}$, we adjust $B$ and $C$, which are adjacent to $\tilde{a}$, and $\p \ra (A,b,c+m,\tilde{a},b+m,c)$. If  $m=B-b$, then we need to adjust $A$ and $C$, and the two consecutive stacks $B$ and $c$ are not changed; in this case  $\p \ra (\tilde{a}+m,b, c+m,\tilde{a}, B, c)$. The case $m=C-c$ follows by symmetry.\\
Case 2:  Let $\p=(A,B,C,a,b,c)$. Again using Lemma~\ref{digsumzero}, we identify the value that needs to be adjusted to create a  digital triangle. If $a$ is the value to be reduced, then we reduce both $A$ and $a$ to $\tilde{a}$, and reduce the other two pairwise maxima to their respective minima, that is $\p \ra (\tilde{a},b,c,\tilde{a},b,c)$. (The case where $c$ needs to be reduced follows by symmetry.) If  $b$ is the value that needs to be reduced to create a zero digital sum, then we reduce $B$ to $\tilde{b} \le b \le B$, and are basically in the situation of Case 1. The two possible legal moves are $\p \ra (a+m, \tilde{b}, c+m, a,b,c)$ with $m=b-\tilde{b}$ or  $\p \ra (a+m, \tilde{b}, C, a,\tilde{b}+m,c)$ with $m=C-c$. 
\end{proof}

The last case for $n=6$ is $\c(6,2)$, which remains an open question. So far, we have not been able to find a conjectured structure for $\L_{\c(6,2)}$ that has not been undone by a counter example. However, we know that the set of losing positions cannot be closed under addition, as $\c(6,2)$ reduces to Nim on three stacks when every other stack has been reduced to zero tokens. In fact, this is the case for all games $\c(n,2)$ for $n \ge 6$.

\section{Larger values of $n$}

As $n$ gets larger, the structure of the losing set becomes more complicated. We show one example for $n=8$, where we have a new feature, namely a minimum involving the sum of stack heights, in the structure of the losing set.

\begin{theorem} \label{8-6} The set of losing positions for the game $\c(8,6)$ is given by 
$$\L=\{ (0,x,a_1,b_1,e,b_2,a_2,x) \mid a_1+b_1=a_2+b_2=x \text{ and } e=\min(x,a_1+a_2)\}.$$
\end{theorem}

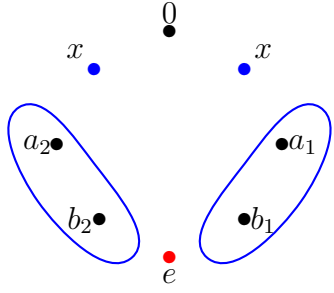
\begin{figure}[htp]
 \begin{center}
  { \psset{xunit=0.5,yunit=0.5}
 \begin{pspicture}(7,7)
 \rput(1,1){{ $\bullet$}} \rput(0.63,1){$b_2$}
 \rput(3,0){$\myred\bullet$}\rput(3,-0.5){$e$}
  \rput(5,5){{\cb $\bullet$}}\rput(5.5,5.5){$x$}
 \rput(6,3){$\bullet$}\rput(6.55,3){$a_1$}
  \rput(5,1){$\bullet$}\rput(5.5,1){$b_1$}
 \rput(3,6){$\bullet$}\rput(3.,6.5){$0$}
  \rput(1,5){{\cb $\bullet$}}\rput(0.5,5.5){$x$}
 \rput(0,3){$\bullet$}\rput(-0.5,3){$a_2$}
 \psccurve[linecolor=blue](4,0)(6.5,1.5)(7,4)(5,2.5)
  \psccurve[linecolor=blue](2,0)(-0.5,1.5)(-1,4)(1,2.5)
 \end{pspicture}}
\caption{Losing positions for $\c(8,6)$.}\label{n8,6}
\end{center}
\end{figure}

\begin{remark} \label{zeros} Before proving Theorem~\ref{8-6} we will discuss the role of the zeros in a losing position. Specifically, we will see that if a losing position has more than one zero then the position will have a reflection symmetry (dotted lines in Figure~\ref{sym}) that clearly shows that any of the zeros can be deemed the $``0"$ of the typical losing position  $\p=(0,x,a_1,b_1,e,b_2,a_2,x) \in \L$.   Note that a zero stack is always between two maximal stacks $x$.
\begin{enumerate} 
\item If $x=0$, then $\p=(0,0,0,0,0,0,0,0)$.
\item If  $a_1=0$, then  $b_1=x$ and $e=\min(x,a_1+a_2)=\min(x,a_2)=a_2$, and therefore, $\p=(0,x,0,x,a_2,b_2,a_2,x)$ and the conditions of $\L$ hold for either of the two zeros as only the sum of $a_2$ and $b_2$ matters, and their positions can be interchanged due to rotational symmetry (see Figure~\ref{sym(2)}).
\item If $ b_1=0$, then $a_1=x$ and $ e=\min(a_1+a_2,x)=\min(x+a_2,x)=x$, and therefore $\p= (0,x,x,0,x,b_2,a_2,x)$, and once more, the conditions can be verified for the second zero as well (see Figure~\ref{sym(3)}).
\item If $e=0$, then either $x=0$, a case considered before, or both $a_1$ and $a_2$ are zero, which results in $\p=(0,x,0,x,0,x,0,x)$, a special case of (2) above.
\end{enumerate}
\end{remark}

\begin{figure}[htp]
  \begin{center}
   \subfigure[Case $(2)$]{\label{sym(2)}
      { \psset{xunit=0.6,yunit=0.6}
 \begin{pspicture}(0,0)(6,6)
 \rput(5,3){{ $\bullet$}} \rput(5.8,3){$0$}
 \rput(4.3,4.3){$\bullet$}\rput(4.9,4.5){$x$}
  \rput(3,5){{$\bullet$}}\rput(3,5.8){$0$}
 \rput(1.7,4.3){$\bullet$}\rput(1.2,4.8){$x$}
  \rput(1,3){$\bullet$}\rput(0.2,3){$a_2$}
 \rput(1.7,1.7){$\bullet$}\rput(1.0,1.7){$b_2$}
\rput(3,1){{ $\bullet$}}\rput(3,0.5){$a_2$}
 \rput(4.3,1.7){$\bullet$}\rput(4.7,1.3){$x$}
 \psline[linestyle=dotted,linewidth=1.5pt,dotsep=2pt](0.5,0.5)(5.5,5.5)
 \end{pspicture}}}  \hspace{2cm}
    \subfigure[Case $(3)$]{\label{sym(3)}
    { \psset{xunit=0.6,yunit=0.6}
 \begin{pspicture}(0,0)(6,6)
 \rput(5,3){{ $\bullet$}} \rput(5.8,3){$x$}
 \rput(4.3,4.3){$\bullet$}\rput(4.8,4.8){$x$}
  \rput(3,5){{$\bullet$}}\rput(3,5.8){$0$}
 \rput(1.7,4.3){$\bullet$}\rput(1.2,4.8){$x$}
  \rput(1,3){$\bullet$}\rput(0.2,3){$a_2$}
 \rput(1.7,1.7){$\bullet$}\rput(1.4,1.4){$b_2$}
\rput(3,1){{ $\bullet$}}\rput(3,0.5){$x$}
 \rput(4.3,1.7){$\bullet$}\rput(4.7,1.3){$0$}
 \psline[linestyle=dotted,linewidth=1.5pt,dotsep=2pt](0,1.8)(6,4.2)
 \end{pspicture}}} 
      \end{center}
  \caption{Symmetric positions.}
  \label{sym}
\end{figure}
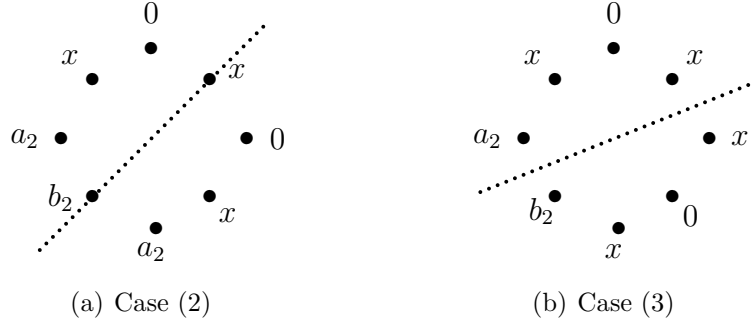

\begin{proof}
Let $S=\{ (0,x,a_1,b_1,e,b_2,a_2,x) \mid a_1+b_1=a_2+b_2=x \text{ and } e=\min(x,a_1+a_2)\}$.  As before, we will show that $S$ is the set of losing positions by showing that $S$ satisfies conditions (I) and (II) of Theorem~\ref{lose}. We start by proving condition (I).

Suppose $\p=(0,x,a_1,b_1,e,b_2,a_2,x)\in S$.  Consider first the case where one of the two $x$ stacks is not reduced by the move. As it is beside a $0$ and remains the maximal element,  neither the other $x$ nor any of the $a_i$ and $ b_i$ with $i=1,2$ can be reduced (otherwise $ a_i+b_i < x$) if the resulting position is to be in $S$. Furthermore, since $e=\min(x,a_1+a_2)$, it too must remain fixed, which implies that there is no legal move in this case. Now assume both $x$ stacks are reduced to $x'<x$. To create a position in $S$,  one must play on at least one of $a_1$ and $b_1$ to make $a_1'+b_1'=x'$ and on at least one of $b_2$ and $a_2$ to make $a_2'+b_2'=x'$. Since it is not possible to play on both $b_1$ and $b_2$  (as no set of six consecutive  stacks contains both of them and the two $x$ stacks), 
at least one of the $a_1$ and $a_2$ stacks  must be reduced. W.l.o.g, assume play is on $a_1$, and thus, since $a_1'+a_2<a_1+a_2=x$ and $x'<x$, we have that  $e' < e$, but   no set of six consecutive  stacks contains both of the $x$ stacks, $a_1$, $e$, and at least one of $a_2$ and $b_2$, and thus there is no legal move. 
This completes the proof for condition (I).

To show condition (II), we will prove that  for different classes of positions $\p$ there is an option $\p'$ of $\p$ that belongs to $S$, and then show that every possible position $\p \notin S$ belongs to at least one of these classes. We will say that $\p$ {\em is solved} if there is a legal move from $\p$ to $\p' \in S$. Furthermore,  in $\c(8,6)$ we have at least two stacks whose height remains the same in a legal move, and we will refer to those stacks as {\em fixed}.

We now prove a sequence of lemmas, each showing that a different class of positions is solved. Lemmas~\ref{valley} and~\ref{max-min} will also be used in  the proofs of the subsequent lemmas.

\begin{definition} If a position $\p$ contains four consecutive stacks $a,b,c$, and $d$ such that $b+c\le \min(a,d)$, then these four stacks are called a {\em valley} of the position, and we will refer to the four stacks satisfying this condition as $\val{a}{d}$.  The {\em size of the valley} is defined as $|\val{a}{d}|=b+c$.
\end{definition}

\begin{lemma} \label{valley} (Valley lemma) A position $\p$ that contains a valley is solved.
\end{lemma}

\begin{proof} 
Consider the position $\p=(a,b,c,d,e,f,g,h)$, and suppose it contains  $\val{a}{ d}$. If there is more than one valley, suppose without loss of generality that $\val{a}{ d}$ has minimal size. We make the $h$ stack  the zero of $\p'$. When reducing the $a$ and $g$ stacks to equal heights we need to consider two cases, namely $\max(f,g) \ge b+c$ and $\max(f,g) < b+c$.

Let $\max(f,g)\ge b+c$. W.l.o.g., $g\ge f$, that is, $g\ge b+c$ (otherwise make $e$ stack the zero of $\p'$) . We fix stacks  $b$ and $c$. Let $a'=b+c$ (possible because $\val{a}{d}$), $h'=0$, and $g'=b+c$. Now, if $e+f\ge b+c$, let $e'$ and $f'$ be such that $e'+f'=b+c$, and $d'=\min(b+c,b+f')$ (possible because $\val{a}{d}$ ensures $d \ge b+c$). The resulting position $\p'=(b+c,b,c,\min(b+c,b+f'),e',f',b+c,0)$ with $e'+f'=b+c$ is in $S$. Note that we always have  $e+f\ge b+c$, as otherwise $ e+f<b+c<d$ and $e+f<b+c\le g$, that is we would have  $\val{d}{g}$ with a smaller size than $\val{a}{d}$, a contradiction to the minimality of $\val{a}{d}$.

Now consider the second case,  $\max(f,g)<b+c$. Then $f<b+c$, $g<b+c$ and w.l.o.g, $f\le g$. We fix $f$ and $g$  and let $h'=0$, $a'=g$ (possible because  $a\ge b+c>g$), $e'=g-f$ (possible because otherwise  $|\val{d}{g}|<|\val{a}{d}|$, which contradicts the minimality of $\val{a}{d}$). Also, $b$ and $c$ are reduced so that $b'+c'=g$, (possible because $b+c>g$), and finally $d'=\min(g,b'+f)$ (possible because  $d\ge a+b>g$).
The resulting position $\p'=(g, b', c', \min(g,b'+f),g-f, f, g, 0)$ with $b'+c'=g$ is in $S$, that is, $\p$ is solved.\end{proof}

\begin{lemma} (Trapezoid lemma) \label{trapezoid}
If the position $\p =(a,b,c,d,e,f,g,h)$  satisfies that $\max(a,h) \le \min(f,c)$,  then $\p$ is solved.
\end{lemma}

\begin{proof} W.l.o.g. assume that $ a\ge h$.  Now if $d+e\le \min(f,c)$, then $\val{c}{f}$ and $\p$ is solved, so we can assume that $d+e>\min(f,c)$. Similarly, if $g+h \le \min(a,f)=a$, then $\val{f}{a}$, so we can assume  that $g+h>a$. With these two inequalities in hand we can proceed: fix $a$ and $h$ and let $b'=0$, $c'=a$, $d'+e'=a$ (possible because $d+e>c\ge a$), $g'=a-h\ge 0$ (possible because $h\le a$ and $g+h>a$), and $f'=\min(a,h+d')$ (possible because $f\ge a$). We then get $\p'=(a,0,a,d',e',\min(a,h+d'),a-h,h) \in S$.
\end{proof}

The next two lemmas consider cases in which adjacent stacks $a$ and $b$ are each smaller than the minimum of a specified pair of stacks.

\begin{lemma} (First double min lemma) \label{dmin1}
A position $\p=(a,b,c,d,e,f,g,h)$ for which $a\le \min(d,f)$ and  $b\le \min(a,g)$ is solved.
\end{lemma}

\begin{proof} By Lemma~\ref{valley}, we only need consider positions $\p$ that do not contain a valley.  We fix 
 $a$ and $b$ and let $e'=0$, $d'=f'=a$ (possible because $a\le \min(d,f)$), $c'=a-b$ (possible because $b\le a$  and $c< a-b$ would imply  $\val{a}{d}$).
For the resulting position to be in $S$, one must  have $g'+h'=a$ and  $a'=a=\min(d',c'+g')=\min(a,a-b+g')$. We have that $g+h\ge a$, otherwise $\val{f}{a}$,  and therefore it is possible  to obtain $g'+h'=a$. Likewise, since $g \ge b$, we can achieve $g' \ge b$. Moreover, both conditions can be satisfied at the same time as follows: if  $g\ge a$, let $g'=a$ and $hÕ=0$, so  $\p'=(a,b, a-b,a,0,a,a,0)$; 
if $g<a$, let $g'=g$, $h'=a-g$ to yield $\p'=(a,b, a-b,a,0,a,g, a-g)$.
\end{proof}

\begin{lemma} (Second double min lemma) \label{dmin2} 
A position $\p=(a,b,c,d,e,f,g,h)$ for which $a\le \min(e,g)$ and $b\le \min(a,d)$ is solved.
\end{lemma}

\begin{proof} There are two cases to be considered, each of which results in a position of the form (2) of Remark~\ref{zeros}. If $b+c\ge a$, we fix $a$ and $b$, and let $c'=a-b$, $d'=b$, $e'=a$, $f'=0$, $g'=a$, and $h'=0$ to obtain $\p'=(a,b, a-b,b,a,0,a,0) \in S$. Otherwise, we fix $b$ and $c$ and let $a'=b+c$, $d'=b$, $e'=b+c$, $f'=0$, $g'=b+c$, and  $h'=0$, which yields $\p'=(b+c,b,c,b,b+c,0,b+c,0) \in S$.
\end{proof}

\begin{lemma}(MaxMin Lemma)\label{max-min} If for a position $\p = (a,b,c,d,e,f,g,h)$, either
\begin{equation}\label{maxmin1eq} \max(b,c,b+c-e) \le \min(f,h, a+b,b+c, (b+c+d)/2)
\end{equation}
or
\begin{equation}\label{maxmin2eq}\max(c,d, c+d-a) \le \min(f,h, c+d, d+e, (b+c+d)/2)
\end{equation}
holds,  then $p$ is solved . 
\end{lemma}

\begin{proof} Let $g'=0$, and keep either $b$ and $c$ or $c$ and $d$ fixed. We first consider the case where $b$ and $c$ are fixed. In order for a legal move to exist,  the following conditions have to be satisfied, where $m$ (the maximum adjacent to the zero of the new position) is a quantity to be determined:
$$ f'=h'=m; \quad a'+b =m; \quad d'+e'=m; \quad \text{ and } c=\min(m,a'+e')=a'+e'.$$

Note that the last equality is an additional assumption used to determine all the values for the new position $\p'$.  If these inequalities can be solved for $m$, then there is a legal move to $\p'=(m-b,b,c, 2m-b-c, b+c-m,m,0,m)$ for each value of $m$ that satisfies the conditions. (Note that we used the assumption that  $c=a'+e'$  to compute $e'$.) All of these entries have to be non-negative, and smaller than the corresponding entries in $\p$. Thus we get two conditions for each of the stacks of $\p'$, which translate into conditions for $m$ as follows:
\begin{eqnarray*}
\begin{tabular}{lll}
$0 \le m \le f$ &  $\Rightarrow$ & $0 \le m \le f$  \\
$0\le m \le h$ & $\Rightarrow$ & $0\le m \le h$ \\
$0\le m-b \le a$ &  $\Rightarrow$ & $b \le m \le a+b$ \\
$c \le m$ & $\Rightarrow$ & $c \le m$ \\
$0 \le 2m-b-c \le d$ &  $\Rightarrow$ & $(b+c)/2 \le m \le (b+c+d)/2$ \\
$0 \le b+c-m \le e$ &  $\Rightarrow$ & $b+c-e \le m \le b+c$ \\
\end{tabular}  
\end{eqnarray*}

Combining these conditions yields $$\max(b,c,(b+c)/2, b+c-e) \le m \le \min(f,h, a+b, b+c, (b+c+d)/2).$$ We can further simplify this condition by recognizing that the average  $(b+c)/2$ is always smaller than either of $b$ and $c$, and thus the average $(b+c)/2$ can be taken out of the maximum requirement, yielding \eqref{maxmin1eq}. The second case, fixing $c$ and $d$, results in Equation \eqref{maxmin2eq} by symmetry across the line through $g$ and $c$.  In this case, $\p'= (c+d-m, 2m-c-d,c,d, m-d,m,0,m)$ is a legal move.
\end{proof}

\begin{example} Suppose that $\p=(4,12,11,9,10,16,1,17)$. Then \eqref{maxmin1eq} is satisfied as $\max(b,c,b+c-e)=\max(12,11,13) =13 \le 16= \min(16,17,16,23,16)=\min(f,h, a+b,b+c, (b+c+d)/2)$. Then for $m = 13, 14, 15, 16$, $\p'=(m-12,12,11, 2m-23, 23-m,m,0,m)$ is a legal move.
\end{example}

We now provide a final lemma which deals with the remaining cases which are small in number, but unfortunately do not fall into a neat unifying structure. 

\begin{lemma} (Clean-up Lemma) \label{cleanup} A position $\p=(a,b,c,d,e,f,g,h)$ for which $ f\ge \min(b,h) \ge \max(d,e)$ and $f\ge c\ge e\ge g$, $d\ge g$ and $a \le \min(c,d)$ is solved.
\end{lemma}

\begin{proof} 
Of the given inequalities, the only  one that does not unambiguously fix the relative order of stack heights  is  $f\ge \min(b,h)$. Several subcases arise which are summarized in Table~\ref{cleanupsum}. In most cases, we will just provide a position $\p'\in S$, and the reader can check that the given position $\p'$ is a legal move, that is $p_i\ge p'_i\ge 0$, using the inequalities of Lemma~\ref{cleanup} and the inequalities of the given subcase. We will provide some details of the proof and remark on the underlying structure for the case $ \min(b,h)=b< d+e-g$. In the first case,  the condition $a+e \le c$ assures that  the choice $c'=\min(m,a+e')\le a+e$ is legal. In each case, $g'=0$ and $f'=h'=m=\min\(h,f,a+b,d+e)$. The other values of $\p'$ are adjusted depending on the value of $m$ so that the resulting position $\p'\in S$. 

In the second case, when $a+e > c$, we apply Lemma~\ref{max-min}. Using~\eqref{maxmin1eq}, a position $\p$ is solved if $\max(b,c,b+c-e) \le \min(f,h, a+b,b+c, (b+c+d)/2)$. Since $b \ge \max(d,e)\ge e$ and $c\ge e$, we have that $b+c-e\ge b$ and $b+c-e\ge c$, so $\max(b,c,b+c-e)=b+c-e$. Thus we have to show that $b+c-e\le \min(f,h, a+b,b+c, (b+c+d)/2)$. But $b+c-e\le  (b+c+d)/2$ is logically equivalent to $ b+c-e \le d+e $, so the condition becomes $m=b+c-e\le m^*=\min(f,h, a+b,b+c, d+e)$. If $m^*=a+b$ or $m^*=b+c$, then $m^*>m$ and ~\eqref{maxmin1eq} is satisfied. If $m^*= f$ or $m^*=d+e$, then for the case $m \le m^*$ ~\eqref{maxmin1eq} is true; if on the other hand $m > m^*$, then we cannot apply Lemma~\ref{max-min}, but give a position $\p'$ using other methods. Finally, if $m^*=h$, we apply ~\eqref{maxmin2eq} of Lemma~\ref{max-min}  which asserts that a position is solved if $\hat{m}=\max(c,d, c+d-a) \le \tilde{m}=\min(f,h, c+d, d+e, (b+c+d)/2)$. Since $a \le \min(c,d)$, we obtain that $\hat{m} =c+d-a$. Also,  $e \le c$ implies $d+e \le d+c$, and together with  $m^*=h$, we have that $h \le \min(f,d+e,d+c)$, which implies that  $\tilde{m}=\min(h,  (b+c+d)/2)$. Furthermore, $c+d-a\le (b+c+d)/2$ is logically equivalent to $c+d-a\le a+b$. Using $m^*=h$ once more, we obtain that $\tilde{m}=\min(h,  a+b)=h$. Thus, \eqref{maxmin2eq} holds if $\hat{m}=c+d-a\le m^*=h$, and $\p'= (c+d-\hat{m}, 2\hat{m}-c-d,c,d, \hat{m}-d,\hat{m},0,\hat{m})=(a,c+d-2a,c,d,c-a,\hat{m},0,\hat{m})$ with $\hat{m}=c+d-a$ is a legal move. Otherwise, if $\hat{m} > m^*$, we can move to $\p'=(a,h-a,a+h-d,d,h-d,h,0,h) \in S$.
\end{proof}

Now that we have all the intermediate results, we can complete the proof of Theorem~\ref{8-6}.  Lemmas~\ref{trapezoid}, ~\ref{dmin1},~\ref{dmin2}, and~\ref{cleanup} are all lemmas of the form: ``If a given set of conditions on the relative sizes of individual stacks holds, then the position is solved". By contrast, Lemmas~\ref{valley} and~\ref{max-min}  were mere tools to prove the other lemmas. Note that there are a total of $7!$ ways to arrange the relative sizes of stacks around the circle, and we need to divide by $2$ to account for the reflection symmetry. Thus there are $7!/2=2520$ different size configurations to be considered. Each 
 of these arrangements is, under a suitable rotation or reflection, covered by at least one of the four main lemmas, as checked by a Visual Basic program that can run on any Excel workbook. (The code can be obtained from \url{http://www.calstatela.edu/faculty/sheubac/}.) It is quite interesting to see how the cases distribute among the four lemmas that settle Theorem~\ref{8-6}. Clearly, Lemma~\ref{trapezoid} is the most powerful, as it covers $2248$ out of the $2520$ cases, roughly $89\%$. Lemma~\ref{cleanup} on the other hand covers only 62 cases, and was specifically designed to cover the 42 cases not already covered by the other lemmas, resulting in the very tedious conditions of Lemma~\ref{cleanup}.  Figure~\ref{dist} shows the contributions of the four lemmas to the proof of Theorem~\ref{8-6}. 
\end{proof} 

\vspace{1cm}

\begin{figure}[htb] \label{dist}
 \begin{center}
 { \psset{xunit=0.5,yunit=0.5}
\begin{pspicture}(19,12)
 \rput(3.5,12.){\cb 1st ``Double Min'' Lemma}
  \rput(6,1){\cg 2nd ``Double Min'' Lemma}
   \rput(12,11){\cp ``Trapezoid'' Lemma}
   \rput(17,4){\color{orange} ``Clean-up'' Lemma}
  \rput(16,1.5){\bf Total = 2520}  
 \rput(3,8){$102$}
 \rput(8,7.5){$452$}
  \rput(5,5.3){$32$}
 \rput(8,5.5){$1100$}
 \rput(10.5,5){$520$}
 \rput(11.5,8){$156$}
 \rput(14.2,7){$20$}
 \rput(17,8){$42$}
 \rput(8,3){$94$}
\psellipse[linecolor=mygreen](7,4)(6,2.5)
\psellipse[linecolor=blue](5.75,7.5)(5,3.75)
\psellipse[linecolor=purple](11,7)(5,3)
\psellipse[linecolor=orange](16,7.5)(3,2.5)
\end{pspicture}}
\end{center}
\caption{The contributions of the various lemmas to the proof of Theorem~\ref{8-6}}
\end{figure}
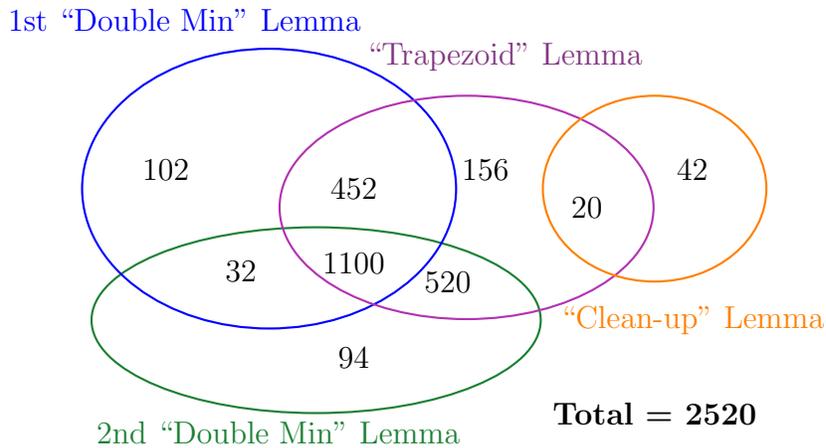

\begin{landscape}
\begin{table}[htb]
\begin{center}
\begin{tabular}{llll}
\multicolumn{3}{c|}{Conditions} &\multicolumn{1}{c}{$\p'$}\\ \hline
\multicolumn{1}{l|}{$\min(h,b)\ge d+e-g$}   & & &  \multicolumn{1}{|l}{  $(0,m,e-g,d,e,m-g,g,m)$}\\ 
\multicolumn{1}{l|}{}   & & &  \multicolumn{1}{|l}{$m=d+e-g$}\\ 
\hline
\multicolumn{1}{l|}{$\min(h,b)=h <  d+e-g$}   & & &  \multicolumn{1}{|l}{  $(0,h,e-g,h+g-e,e,h-g,g,h)$}\\ \hline
$\min(h,b)=b <  d+e-g$ &  \multicolumn{1}{|l}{ $a+e \le c $} &  \multicolumn{1}{|l}{ $m=a+b$} &   \multicolumn{1}{|l}{ $(a,b,c',d',e',m,0,m)$}\\
&  \multicolumn{1}{|l|}{$m=\min(h,f,a+b,d+e)$} &\multicolumn{1}{l|}{}  & $d'+e'=m$, $c'=\min(m,a+e')$ \\ \cline{3-4}
&  \multicolumn{1}{|c|}{} &  \multicolumn{1}{l} {$m=d+e$}&  \multicolumn{1}{|l}{ $(a',b',c',d,e,m,0,m)$}\\
&  \multicolumn{1}{|c|}{} &  \multicolumn{1}{l|} {}& $a'+b'=m$, $c'=\min(m,a+e')$\\ \cline{3-4}
&  \multicolumn{1}{|c|}{} & \multicolumn{1}{l} {$m=f$}&  \multicolumn{1}{|l}{ $(a',b',c',m-e,e,m,0,m)$}\\ 
&  \multicolumn{1}{|c|}{} &  \multicolumn{1}{l|} {}& $a'+b'=m$, $c'=\min(m,a+e')$\\ \cline{3-4}
&  \multicolumn{1}{|c|}{} &  \multicolumn{1}{l} {$m=h$}&  \multicolumn{1}{|l}{ $(a,m-a,c',d',e',m,0,m)$}\\
&  \multicolumn{1}{|c|}{} &  \multicolumn{1}{l|} {}& $d'+e'=m$, $c'=\min(m,a+e')$\\ \cline{2-4}
&  \multicolumn{1}{|l}{ $a+e > c $} &  \multicolumn{1}{|l}{ $m^*=a+b$ or } &   \multicolumn{1}{|l}{ $(c-e,b,c, m-e, e,m,0,m)$}\\
&  \multicolumn{1}{|l}{ $m^*=\min(h,f,a+b,b+c,d+e)$} &  \multicolumn{1}{|l}{ $m^*=b+c$ or} &   \multicolumn{1}{|l}{ }\\
&   \multicolumn{1}{|l|} {$m=b+c-e$}&  \multicolumn{1}{l}{$m\le m^*=f $ or} &   \multicolumn{1}{|l}{ }\\ 
&   \multicolumn{1}{|l|} {}&  \multicolumn{1}{l}{$m\le m^*=d+e $} &   \multicolumn{1}{|l}{ }\\ \cline{3-4}
&   \multicolumn{1}{|l|} {}&  \multicolumn{1}{l}{$m>m^*=f $} &   \multicolumn{1}{|l}{$(f-b,b,e+f-b,f-e,e,f,0,f)$ }\\ \cline{3-4}
&   \multicolumn{1}{|l|} {}&  \multicolumn{1}{l}{$m>m^*=d+e$} &   \multicolumn{1}{|l}{$(a',b',c',d,e,m^*,0,m^*)$ }\\ 
&   \multicolumn{1}{|l|} {}&  \multicolumn{1}{l}{} &   \multicolumn{1}{|l}{$b'=\min(b,m^*),a'+b'=m^*$ }\\ 
&   \multicolumn{1}{|l|} {}&  \multicolumn{1}{l}{} &   \multicolumn{1}{|l}{$c=\min(m^*,e+a')$ }\\  \cline{3-4}
&   \multicolumn{1}{|l|} {$\hat{m}=c+d-a$}&  \multicolumn{1}{l}{$\hat{m}\le m^*=h$} &   \multicolumn{1}{|l}{$(a,c+d-2a,c,d,c-a,\hat{m},0,\hat{m})$ }\\  \cline{3-4}
&   \multicolumn{1}{|l|} {}&  \multicolumn{1}{l}{$\hat{m}>m^*=h$} &   \multicolumn{1}{|l}{$(a,h-a,a+h-d,d,h-d,h,0,h)$ }\\ 
\hline
\multicolumn{4}{c}{}\\
\end{tabular}
\end{center}
\caption{Legal moves for the subcases of Lemma~\ref{cleanup} \label{cleanupsum}}
\end{table}

\end{landscape}

\section{Generalizations}\label{GenRes}

Even though we have given results for a number of values of $n$, it would  clearly be more satisfying to obtain general results that go beyond the ``extreme cases'' $\c(n,1)$, $\c(n,n)$, and $\c(n,n-1)$.  Ehrenborg and Steingr{\'{\i}}msson~\cite{EhrSte1996} and Horrocks~\cite{Hor2010} investigated a more general set of games, namely playing Nim on a simplicial simplex, and obtained structural results for the set of losing positions. In particular, the results in~\cite{EhrSte1996} contain $\c(5,2)$ and $\c(5,3)$, while~\cite{Hor2010} contains $\c(6,3)$ as a special case. The question then becomes whether these results solve $\c(n,k)$ for other values of $n$. We consider this to be unlikely, as the structural results in~\cite{EhrSte1996} and~\cite{Hor2010} are linear in nature while our results for $\c(6,2)$, $\c(6,4)$ and $\c(8,6)$ contain non-linear elements like digital sum and minimum. Nevertheless, an investigation of the structure of the circuits of $\c(n,k)$,  which are at the heart of the results of Ehrenborg, Steingr{\'{\i}}msson, and Horrocks, might yield additional insights. We start by defining the necessary terminology, adapting the definitions given in~\cite{EhrSte1996} for the special case of circular Nim.

\begin{definition} \label{simp comp} A {\em simplicial complex} $\Delta$ on a finite set of nodes $V=\{1,2,\ldots,n\}$ is a collection of subsets of $V$ such that $\{v\} \in \triangle$ for every $v \in V$, and $B \in \triangle$ whenever $A \in \triangle$ and $B \subseteq A$. The elements of $\Delta$ are called {\em faces} and represent the choices for the stacks from which  a player can take tokens.  A face that is maximal with respect to inclusion is called a {\em facet}. A minimal (with respect to  inclusion) non-face of $\Delta$ is called a {\em circuit}. The {\em size of a circuit} is  the number of nodes in the circuit.
\end{definition}

For circular Nim $\c(n,k)$, the simplicial complex is given by 
$$\triangle =\bigcup_ {i =1}^n \bigcup _{j=0}^{k-1}\{i,( i+1), \ldots, (i+j)\}\imod{n}.$$ The facets are the sets consisting of $k$ consecutive vertices, while the structure of circuits  is harder to describe in general. (They are not the sets consisting of $k+1$ consecutive vertices.) However, for $k=2$ we can explicitly describe the circuits and can enumerate them as well since the structure of the circuits is very simple in this case. 

\begin{lemma} The circuits of $\c(n,2)$ are of the form $\{i,j\}$ with  $i=1,2, \ldots, n, j=i+2,\ldots, i-2 \imod{n}$. The number of circuits of $\c(n,2)$ is given by $n(n-3)/2$.
\end{lemma}

\begin{proof} By definition, a circuit is a set of stacks on which play is not allowed, but all subsets of the circuit are allowed choices for the stacks. Thus, for any stack $i$, all but the two pairs consisting of $i$ and its immediate neighbors, $i+1$ and $i-1$ form a circuit. There are $n-3$ such choices for each of the $n$ stacks; division by $2$ takes into account symmetry. 
\end{proof}

We now prove a result on the size of the circuits of $\c(n,k)$ for any $n$ and $k$. Before we can do so, we need a few definitions. In what follows we always assume that vertices are listed in increasing (clockwise) order and that indices are given modulo $n$.

\begin{definition} An {\em arc of length  $m$} with {\em end vertices} $i$ and ${i+m}$ is a set of $m+1$ consecutive vertices $\{i,{i+1},\ldots,{i+m}\} \subseteq \{1,2,\ldots,n\}$. We denote the arc from  $i$ to ${i+m}$ by $\arc{i}{(i+m)}$.  With each set $V$ of vertices, we associate two measurements: the {\em size (= number of elements)} of the set $V$, denoted by $|V|$, and the {\em span size} $sp(V)$, which is the length of the smallest arc containing $V$. 
\end{definition}

Note that $sp(V) \ge |V|-1$, with equality exactly when the elements of $V$ are consecutive vertices.

\begin{remark} \label{arc}
\begin{enumerate}
\item The smallest arc containing  a set is not necessarily unique, but its length is. For example, if $n$ is even, and the set $V$ consists of two diagonally opposite vertices, then we have two arcs of length $n/2$.
\item If an arc that covers a given set $V$ is minimal, then the arc's two end vertices belong to $V$; for if not, one could obtain a smaller arc that covers $V$ by simply removing the end vertex that does not belong to $V$, thus reducing the length of the arc.
\item In the $\c(n,k)$ game, every face is contained in an arc of length at most $k-1$, and every set with span size at most $k-1$ is a face.
\end{enumerate}
\end{remark}
 
\begin{definition} Given a set $V=\{v_1, v_2, \ldots, v_m\}$ (where the $v_i$ appear in clockwise order), the {\em distance set} $D_V=\{d_1,d_2, \ldots, d_m\}$ is the set of lengths of the arcs $\arc{v_i}{v_{i+1}}$, that is $ d_i = v_{i+1}- v_i$ for $i=1,\ldots,m-1$ and $d_m=v_1+n-v_m$.
\end{definition}

Note that the sum of the distances of any distance set of $V\subseteq \{1,2,\ldots,n\}$ equals $n$, the total number of vertices of a $\c(n,k)$ game, { and that each set of $m$ distances  $d_1, d_2, \ldots, d_{m}$ with $0 < d_i\le n$ and $\sum_{i=1}^m d_i =n$ uniquely describes an $m$-subset of $\{1,2,\dots,n\}$  (up to rotation).}
 
\begin{example} Let $V=\{2,5,6\}$ and $n=8$. Then $d_1=3$, $d_2=1$, and $d_3=4$, so  $D_V=\{3, 1, 4\}$. Adding the distances gives $3+1+4 = 8$.
\end{example}

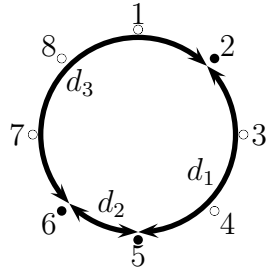
\begin{figure}[htp]
 \begin{center}
  \begin{pspicture}(-1.8,-1.8)(1.8,1.8)
\psset{unit=0.7cm}
  \psdots*[dotstyle=o](0,2)(2,0)(1.45,-1.45)(-2,0)(-1.45,1.45)
  \psdots*[dotstyle=*](1.45,1.45)(0,-2)(-1.45,-1.45)
 
  \psarc[linewidth=2pt]{<->}(0,0){1.85}{270}{45}
  \psarc[linewidth=2pt]{<->}(0,0){1.85}{225}{270} 
  \psarc[linewidth=2pt]{<->}(0,0){1.85}{45}{225}
  \rput(1.2,-0.7){$d_1$}  \rput(-0.5,-1.3){$d_2$}  \rput(-1.1,1){$d_3$}
  \rput(0,2.3){$1$} \rput(0,-2.3){$5$} \rput(-2.3,0){$7$} \rput(2.3,0){$3$} \rput(1.7,1.7){$2$} \rput(1.7,-1.7,){$4$}
  \rput(-1.7,-1.7,){$6$}\rput(-1.7,1.7,){$8$}

  \end{pspicture}
\caption{Distances between vertices.}\label{distances}
\end{center}
\end{figure}

\begin{theorem} \label{circuit cond} A set of vertices $V=\{v_{1}, v_{2}, \ldots, v_{\ell}\}$ is a circuit of $\c(n,k)$ if and only if the  following conditions hold on its distance set $D_V$:
\begin{enumerate} 
\item $d_i+d_{i+1} > n-k$   
and 
\item $d_i \le n-k$
\end{enumerate}
for $ i=1,2,\ldots,\ell$, where $d_{\ell+1}= d_1$.
\end{theorem}

\begin{proof} $``\Rightarrow"$ Suppose $V$ is a circuit. Then  by Remark~\ref{arc} (since $V$ is not a face),
$$ k \le sp(V)=\min_i\{n-d_i\}\le n-d_i \quad \forall i=1\ldots,\ell,$$
so (2) is satisfied. Also, $V\backslash\{v_i\}$ is a face for all $i$, which implies that $sp(V\backslash\{v_i\})\le k-1$. Since $V$ is not a face, the minimal arc covering $V\backslash\{v_i\}$ has to be the arc $\arc{v_{i+1}}{v_{i-1}}$ (otherwise, the minimal arc would also include $v_i$, and thus $V$ would be a face). Therefore, 
$$k-1 \ge sp(V\backslash\{v_i\})=|\arc{v_{i+1}}{v_{i-1}}|=n-d_i-d_{i+1},$$
so (1) holds. 

$``\Leftarrow"$ Assume conditions (1) and (2) hold. Since $d_i \le n-k$, $$sp(V)=\min_i|\arc{v_{i+1}}{v_{i}}|=\min_i\{n-d_i\}\ge k,$$ and so $V$ is not a face. Also, 
\begin{eqnarray*}
sp(V\backslash\{v_i\})&=&\min_{j \ne i}\{|\arc{v_{j+1}}{v_j}|,|\arc{v_{i+1}}{v_{i-1}}|\}\\
&=&\min_{j \ne i}\{\underbrace{n-d_j}_{\ge k},\underbrace{n-d_i-d_{i+1}}_{<k}\}=n-d_i-d_{i+1}\le k-1,
\end{eqnarray*}
so $V\backslash\{v_i\}$ is a face for every $i$, and thus $V$ is a circuit. This completes the proof.
\end{proof}

\begin{theorem} \label{circuit length} For $\c(n,k)$ with $n>1$ and $1<k<n$, a circuit of length $\ell$ exists if and only if 
\begin{equation}\label{circlen} \frac{n}{s} \le \ell \le \frac{2n}{s+1}
\end{equation}
where $s=n-k$.
\end{theorem}

 Table~\ref{cirlen} shows the size of circuits for given $n$ and $k$.

\begin{table}[htdp] 
\begin{center}

{\small
$\begin{array}{c|ccccccccc}
  &k= 2 & 3 & 4 & 5 & 6 & 7 & 8 & 9 & 10\\ \hline
n= 3 & 
\begin{array}{c}
 3
\end{array}
 &  &  &  &  &  &  &  &\\
 4 & 
\begin{array}{c}
 2
\end{array}
 & 
\begin{array}{c}
 4
\end{array}
 &  &  &  &  &  & & \\
 5 & 
\begin{array}{c}
 2
\end{array}
 & 
\begin{array}{c}
 3
\end{array}
 & 
\begin{array}{c}
 5
\end{array}
 &  &  &  &  &  &\\
 6 & 
\begin{array}{c}
 2
\end{array}
 & 
\begin{array}{c}
\{2,3\}
\end{array}
 & 
\begin{array}{cc}
\{3,4\}
\end{array}
 & 
\begin{array}{c}
 6
\end{array}
 &  &  &  & & \\
 7 & 
\begin{array}{c}
 2
\end{array}
 & 
\begin{array}{c}
 2
\end{array}
 & 
\begin{array}{c}
 3
\end{array}
 & 
\begin{array}{c}
 4
\end{array}
 & 
\begin{array}{c}
 7
\end{array}
 &  &  &  & \\
 8 & 
\begin{array}{c}
 2
\end{array}
 & 
\begin{array}{c}
 2
\end{array}
 & 
\begin{array}{cc}
\{2,3\}
\end{array}
 & 
\begin{array}{cc}
\{3,4\}
\end{array}
 & 
\begin{array}{cc}
\{4,5\}
\end{array}
 & 
\begin{array}{c}
 8
\end{array}
 &  &  &\\
 9 & 
\begin{array}{c}
 2
\end{array}
 & 
\begin{array}{c}
 2
\end{array}
 & 
\begin{array}{cc}
\{2,3\}
\end{array}
 & 
\begin{array}{c}
 3
\end{array}
 & 
\begin{array}{cc}
\{3,4\}
\end{array}
 & 
\begin{array}{cc}
\{5, 6\} 
\end{array}
 & 
\begin{array}{c}
 9
\end{array}
 & & \\
 10 & 
\begin{array}{c}
 2
\end{array}
 & 
\begin{array}{c}
 2
\end{array}
 & 
\begin{array}{c}
 2
\end{array}
 & 
\begin{array}{cc}
\{2,3\}
\end{array}
 & 
\begin{array}{cc}
\{3,4\}
\end{array}
 & 
\begin{array}{cc}
\{4,5\}
\end{array}
 & 
\begin{array}{cc}
\{5, 6\} 
\end{array}
 & 
\begin{array}{c}
 10
\end{array} &\\
\vdots &\vdots &\vdots &\vdots &\vdots & \vdots&\vdots &\vdots&\vdots&\\
15 & 2 & 2& 2& 2& 
\begin{array}{cc}
\{2,3\}
\end{array}
&
\begin{array}{cc}
\{2,3\}
\end{array}
& 3& 
\begin{array}{cc}
\{3, 4\}
\end{array} &
\begin{array}{cc}
\{3, 4, 5\}
\end{array}\\
\multicolumn{10}{c}{}
\end{array}$}
\end{center}
\caption{Possible lengths of circuits for given $n$ and $k$\label{cirlen} }

\end{table}

{ Before giving a proof of Theorem~\ref{circuit length} we point out that Table~\ref{cirlen} shows that the conditions of  Horrocks~\cite{Hor2010} are unlikely to be satisfied except in the very special case of $\c(6,3)$. Theorem 21 requires that the circuits split into two sets, each of which is a partition of $\{1, 2, \ldots, n\}$ with specific conditions for each vertex. However, in most cases, the sizes of the circuits do not allow for such partitions, without even considering whether the vertex condition is satisfied. For example, for $\c(6,4)$, the circuit sizes are $3$ and $4$, and therefore, the circuits cannot create two partitions of $\{1, 2, \ldots, 6\}$. Obviously, this is not a rigorous proof that there is no instance in which the conditions of Theorem 21~\cite{Hor2010} are satisfied, but it corroborates what we have seen for $\c(6,4)$ and $\c(8,6)$, namely, that the set of losing positions is no longer a linear combination of some basis elements.  }

\begin{proof} $``\Rightarrow"$ Let $V=\{v_1,v_2,\ldots, v_{\ell}\}$ be a circuit. Then $d_i+d_{i+1} > s\ge s+1$ and $d_i \le s$ for all $i$. Summing over $i=1,\ldots, \ell$, we obtain 
$$ \sum_{i=1}^{\ell}d_i+d_{i+1}\ge\ell(s+1) \quad \text{and}\quad \sum_{i=1}^{\ell} d_i \le \ell\cdot s.$$
Since $d_{\ell+1}=d_1$ and $\sum_{i=1}^{\ell}d_i=n$, we have that 
$$ 2n\ge\ell (s+1)\quad \text{and}\quad n \le \ell\cdot s,$$
and solving for $\ell$ gives the desired inequalities.

$``\Leftarrow"$ To show that there are circuits of the given lengths we will exhibit a circuit for the lower and the upper bounds, and then provide an algorithm to create a circuit of any intermediate length from the circuit for the upper bound. 
Note that for $s=1$, the only circuit consists of all vertices (and thus has size $n$) as each subset of size $n-1$ or smaller is a face. For $s \ge 2$ and $n=m \cdot s +r$ with $0 \le r<s$,  we will show that the set $C=\{1, s+1, 2s+1, ....,m\cdot s+1\}\imod{n}$ is a circuit of size $\ell_1=\lceil n/s \rceil$. Note that $\ell_1 = m$ if $r=0$, and $\ell_1 = m+1$ if $r > 0$. Figure~\ref{circll} shows the construction for the circuit of size $\ell_1$, where the black dots indicate the vertices that make up the circuit, and the left and right end of the string of $n$ vertices are connected in the circular arrangement.

\begin{figure}[htp]
 \begin{center}
  \begin{pspicture}(-1,0.2)(10,5)
  { \psset{xunit=1.5,yunit=1.5}
\multiput(0,0.75)(0,-2){2}{  
\multiput(0.1,2.5)(1.7,0){2}{\spic}\rput(3,1.7){$\ldots$}\rput(3.75,1.7){\spic}
\pscurve(0.09,1.8)(0.55,1.95)(1.2,1.8)\rput(0.55,2.1){$d_1$}
\pscurve(1.2,1.8)(1.65,1.95)(2.3,1.8)\rput(1.65,2.1){$d_2$}
\psline(3.6,1.85)(3.75,1.8)
\pscurve(3.75,1.8)(4.25,1.95)(4.9,1.8)\rput(4.25,2.1){$d_m$}}

\pscurve(4.9,0.95)(5.35,1.1)(5.75,1)\rput(5.4,1.3){$d_{m+1}$}
\psdots*[dotstyle=*](4.9,0.87)
\psdots*[dotstyle=o](5.15,0.87)(5.6,0.87)
\rput(5.35,0.87){...}
\psline(4.8,0.8)(4.8,0.77)(5.7,0.77)(5.7,0.8)
\rput(5.3,0.6){{\tiny$r \ge 1$}}
\rput(-1,1.5){or}
} \end{pspicture}
\caption{Circuit construction for lower limit of $\ell$.}\label{circll}
\end{center}
\end{figure}
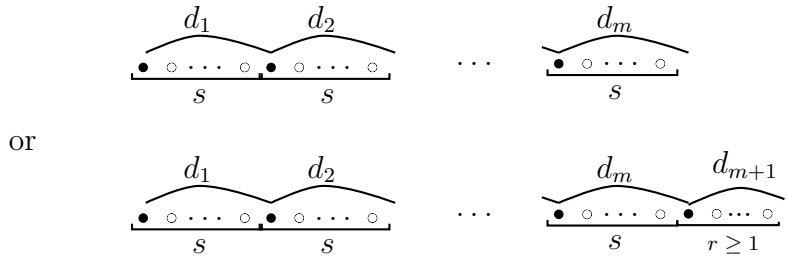

We need to show that $d_i+d_{i+1}>s$ and that $d_i\le s$. By construction of $C$, $d_i=s$ for $i=1,\ldots, m$, and if  $r >0$, then $d_{m+1}=r<s$. In addition, $d_i+d_{i+1}$ is either $2s$ or $s+r$ with $r \ge 1$, so $C$ is a circuit of size $\ell_1$. 

To create a circuit of size $\ell_2=\lfloor 2n/(s+1) \rfloor$, we spread out roughly twice as many vertices as evenly as possible, except potentially for the last one. Let $n=m' (s+1) + r'$,  $\hl=\lfloor (s+1)/2\rfloor$ and $\hu=\lceil (s+1)/2 \rceil$. Then  $\hl+\hu=s+1$.  Let $$C'=\{k(s+1)+1, k(s+1)+1+\hu \mid k=0,\ldots, m'-1\}, $$ that is, we start at  vertex $1$ and then alternate distances $\hu$ and $\hl$ (or have an exact even spread if $\hl=\hu$). As in the case of the lower bound, we have to define what happens in the case when $r' >0$. In fact, we need to distinguish between $r'<(s+1)/2$ and $r'\ge (s+1)/2$, since those two cases distinguish between $\ell_2=2m'$ and $\ell_2=2m'+1$. We claim that a circuit of size $\ell_2$ is given by 
$$\left\{\begin{array}{ll}C' & \text{if } r'<(s+1)/2 \\C''=C' \cup \{m'(s+1)+1\}& \text{if } r'\ge (s+1)/2\end{array}\right..$$

For example, if $n=31$ and $s=4$, then the circuit vertices are given by $C'=\{1,4,6,9,11,14,16,19,21,24,26,29\}$, while for $n=34$ and $s=4$, the circuit vertices consist of $C' \cup \{31\}$. By construction of $C'$ and $C''$, the first $2m'-1$ distances in  $D_{C'}$ and $D_{C''}$ are alternating between $\hu$ and $\hl$, and therefore, $d_i \le s$ (for $s \ge 1$ or $k<n$) and $d_i+d_{i+1}=s+1>s$ satisfying the circuit conditions independent of the value of $r'$. We now consider the two cases $r'<(s+1)/2$ and $r'\ge (s+1)/2$ separately to show the required inequalities for the remaining vertices.

 If $r'<(s+1)/2$, then $D_{C'}=\{\hu,\hl,\ldots,\hl,\hu,\hl+r'\}$ as $d_{2m'}=n+v_1-v_{2m'}$. Considering the two possibilities  $(s+1)/2\in \mathbb{N} $ and $(s+1)/2\notin \mathbb{N} $ separately, it can be shown that  $d_{2m'}=\hl+r'\le s$. Also, $d_{2m'-1}+d_{2m'}=d_1+d_{2m'}=\hu+(\hl+r')\ge s+1$. 
If on the other hand $r'\ge (s+1)/2$, then $D_{C''}=\{\hu,\hl,\ldots,\hl,\hu,\hl,r'\}$ with $r' \le s$, $d_{2m'}+d_{2m'+1}=\hl+r'\ge\hl+\hu=s+1$, and $d_{2m'+1}+d_1=r'+\hu\ge\hu+\hu\ge s+1$. Thus, $C'$ and $C''$ are circuits of size $\ell_2$.

To show the existence of circuits of intermediate length, we transform the circuit for the upper limit step by step into the circuit for the lower limit, reducing the number of vertices by one in each step. We now describe the algorithm, and will use the term {\em $s$-barrier (for the $i^{\text{th}}$ segment)} to denote the space between vertices $(i-1)\cdot s$ and $(i-1)\cdot s+1$ for $i\ge 1$.

\begin{itemize}
\item[Step 0:] Starting with the circuit of size $\ell_2$ described above, divide the vertices into segments of $s$ vertices. If the rightmost (partial) segment does not contain a circuit vertex, then move the rightmost circuit  vertex of the next-to-last segment into the last segment next to that segment's $s$-barrier. If this process leaves the next-to-last segment without a circuit vertex, then move the rightmost circuit vertex of the adjacent segment (on the left) into the next-to-last segment adjacent to its $s$-barrier. (Note that the two rightmost full segments of length $s$ have to contain at least three circuit vertices.)
\item[Step 1:]  Search for the leftmost  segment that does not consist of a single circuit vertex next to its $s$-barrier. 
\begin{itemize} 
\item[$\bullet$] If the segment contains two vertices, move (if needed) the left one to its $s$-barrier, and delete the second one. Move the leftmost circuit vertex in the next segment to its $s$-barrier and move all circuit vertices in the adjacent segments by the same amount to the left until there is a segment where the circuit vertex would have to cross an $s$-barrier. In that segment, move the left circuit vertex to its $s$-barrier and move the other circuit vertex (if any) to the left by the same amount as the leftmost vertex in that segment.
\item[$\bullet$] If the segment contains only one circuit vertex, move it to its $s$-barrier and repeat Step 1.
\end{itemize}
\item[Step 2:]  Repeat Step 1 until all segments consist of a single circuit vertex next to their respective $s$-barrier. 
\end{itemize}

We visualize this algorithm in Figure~\ref{alg} for the case of $n=31$ and $s=4$, that is, the game $\c(31,26)$. In this case, $\hu=3$ and  $\hl=2$, and circuits of sizes $\ell = 8,9,10,11$ and $12$ need to be displayed. We start with the configuration of circuit vertices at positions $5k+1$ and $5k+4$ for $k=0,\ldots, 5$. Positions with circuit vertices are displayed as black dots, other positions are displayed as open circles, and  $s$-barriers are displayed as dotted lines. For each reduction step, only the  segments where circuit vertices change positions are displayed, followed by the resulting circuit.

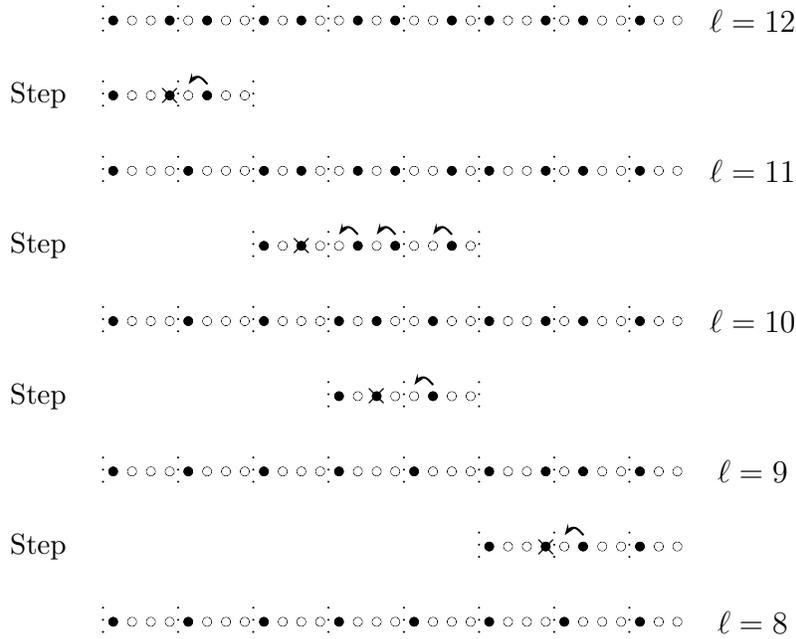
\begin{figure}[htp]
 \begin{center}
  \begin{pspicture}(0,-1)(12,9)
  { \psset{xunit=1,yunit=1}
\multiput(2,0)(.25,0){31}{\psdots*[dotstyle=o](0,0)}  
\multiput(2,0)(1,0){8}{\psdots*[dotstyle=*](0,0)}
\multiput(-0.012,0)(1,0){8}{\psline[linestyle=dotted](1.875,-0.15)(1.875,0.2)}
\rput(10.5,0){$\ell = 8$}

\multiput(2,2)(.25,0){31}{\psdots*[dotstyle=o](0,0)}  
\multiput(2,2)(1,0){6}{\psdots*[dotstyle=*](0,0)}
\put(7.75,2){\psdots*[dotstyle=*](0,0)}\put(8.25,2){\psdots*[dotstyle=*](0,0)}\put(9,2){\psdots*[dotstyle=*](0,0)}
\multiput(-0.012,2)(1,0){8}{\psline[linestyle=dotted](1.875,-0.15)(1.875,0.2)}
\rput(10.5,2){$\ell = 9$}

\multiput(2,4)(.25,0){31}{\psdots*[dotstyle=o](0,0)}  
\multiput(2,4)(1,0){4}{\psdots*[dotstyle=*](0,0)}
\put(5.5,4){\psdots*[dotstyle=*](0,0)}\put(6.25,4){\psdots*[dotstyle=*](0,0)}\put(7,4){\psdots*[dotstyle=*](0,0)}
\put(7.75,4){\psdots*[dotstyle=*](0,0)}\put(8.25,4){\psdots*[dotstyle=*](0,0)}\put(9,4){\psdots*[dotstyle=*](0,0)}
\multiput(-0.012,4)(1,0){8}{\psline[linestyle=dotted](1.875,-0.15)(1.875,0.2)}
\rput(10.5,4){$\ell = 10$}

\multiput(2,6)(.25,0){31}{\psdots*[dotstyle=o](0,0)}  
\multiput(2,6)(1,0){3}{\psdots*[dotstyle=*](0,0)}
\put(7,6){\psdots*[dotstyle=*](0,0)}
\put(7.75,6){\psdots*[dotstyle=*](0,0)}\put(8.25,6){\psdots*[dotstyle=*](0,0)}\put(9,6){\psdots*[dotstyle=*](0,0)}
\multiput(-0.012,6)(1,0){8}{\psline[linestyle=dotted](1.875,-0.15)(1.875,0.2)}
\multiput(4.5,6)(1.25,0){2}{\psdots*[dotstyle=*](0,0)}\multiput(5.25,6)(1.25,0){2}{\psdots*[dotstyle=*](0,0)}\rput(10.5,6){$\ell = 11$}

\multiput(2,8)(.25,0){31}{\psdots*[dotstyle=o](0,0)}  
\multiput(2,8)(1.25,0){6}{\psdots*[dotstyle=*](0,0)}
\multiput(2.75,8)(1.25,0){6}{\psdots*[dotstyle=*](0,0)}
\multiput(-0.012,8)(1,0){8}{\psline[linestyle=dotted](1.875,-0.15)(1.875,0.2)}
\rput(10.5,8){$\ell = 12$}

\multiput(2,7)(.25,0){8}{\psdots*[dotstyle=o](0,0)}  
\rput(2.75,7){\psdots*[dotstyle=*](0,0)}\rput(2.75,7){$\times$}
\pscurve{<-}(3,7.15)(3.125,7.25)(3.25,7.15)
\multiput(2,7)(1.25,0){2}{\psdots*[dotstyle=*](0,0)}
\multiput(-0.012,7)(1,0){3}{\psline[linestyle=dotted](1.875,-0.15)(1.875,0.2)}
\rput(1,7){{\small Step}}

\multiput(4,5)(.25,0){12}{\psdots*[dotstyle=o](0,0)}  
\rput(4.5,5){$\times$}
\pscurve{<-}(5,5.15)(5.125,5.25)(5.25,5.15)
\pscurve{<-}(5.5,5.15)(5.625,5.25)(5.75,5.15)
\pscurve{<-}(6.25,5.15)(6.375,5.25)(6.5,5.15)
\multiput(4,5)(1.25,0){3}{\psdots*[dotstyle=*](0,0)}\multiput(4.5,5)(1.25,0){2}{\psdots*[dotstyle=*](0,0)}
\multiput(1.988,5)(1,0){4}{\psline[linestyle=dotted](1.875,-0.15)(1.875,0.2)}
\rput(1,5){{\small Step}}

\multiput(5,3)(.25,0){8}{\psdots*[dotstyle=o](0,0)}  
\rput(5.5,3){$\times$}
\pscurve{<-}(6,3.15)(6.125,3.25)(6.25,3.15)
\multiput(5,3)(1.25,0){2}{\psdots*[dotstyle=*](0,0)}\rput(5.5,3){\psdots*[dotstyle=*](0,0)}
\multiput(2.988,3)(1,0){3}{\psline[linestyle=dotted](1.875,-0.15)(1.875,0.2)}
\rput(1,3){{\small Step}}

\multiput(7,1)(.25,0){11}{\psdots*[dotstyle=o](0,0)}  
\rput(7.75,1){$\times$}
\pscurve{<-}(8,1.15)(8.125,1.25)(8.25,1.15)
\multiput(7,1)(1.25,0){2}{\psdots*[dotstyle=*](0,0)}\multiput(7.75,1)(1.25,0){2}{\psdots*[dotstyle=*](0,0)}
\multiput(4.988,1)(1,0){3}{\psline[linestyle=dotted](1.875,-0.15)(1.875,0.2)}
\rput(1,1){{\small Step}}

} \end{pspicture}
\caption{Circuit construction for intermediate values  of $\ell$.}\label{alg}
\end{center}
\end{figure}

Each complete  application of Step 1 reduces the number of circuit vertices by one, and hence $\ell$ by one. Furthermore, the distance conditions for circuits remain intact. The initial rearrangement of vertices (if needed) at the right end creates distances that are at most $s$. In addition, the distances do not decrease, so condition (1) remains intact. Now let's look at the distances in the segments where vertices are moved or deleted. We proceed from left to right. 

In the segment in which the vertex is deleted, the distance between the vertices adjacent to the deleted vertex is exactly $s$. The vertices to the right of the deleted vertex  that moved by the same amount to the left  retain their relative distances. The rightmost of these vertices and the circuit vertex to its right that moved a smaller distance due to the non-crossing of barriers have a distance that is at most $s$. That circuit vertex and its neighbor to the right either maintain their distance (if in the same segment)  or their distance increases to at most $s$ (if  in different segments). Likewise, the sums of consecutive distances continues to be at least $s+1$.

This construction shows that circuits of all required sizes exist (they are obviously not unique), therefore completing the proof.

\end{proof}

We will also provide a second proof for Theorem~\ref{circuit length} which uses a purely algebraic approach to show that the circuit  conditions are satisfied when the stacks are spread out as equally as possible. The proof will involve many floor and ceiling functions as the stacks are integer distances apart. We provide three useful lemmas that will aide in the algebraic proof of Theorem~\ref{circuit length}.

\begin{lemma}[Reciprocities of ceilings]\label{rec}
If $x$, $y$ and $n$ are any three integers, then
$$\ceil{\frac{n}{x}} \le y \iff \ceil{\frac{n}{y}} \le x. $$
\end{lemma}

In the proof, we will repeatedly use the fact that $\lceil x \rceil$   is a non-decreasing function and thus, for a fixed $n$,  $x \le y$ implies  $ \lceil\frac{n}{y}\rceil \le  \lceil\frac{n}{x}\rceil$. 

\begin{proof}
If $n=k\cdot x$, then $\ceil{\frac{n}{x}}=k$  and $\ceil{\frac{n}{x}}=k \le y$ implies that $\lceil\frac{n}{y}\rceil \le \ceil{\frac{n}{k}} =x$. Suppose now that $n$ is not a multiple of $x$, that is, $n=k\cdot x+m$, for a positive $m < x$, and $\lceil \frac{n}{x}\rceil =k+1$. Note that $n=(k+1)x-(x-m)$, which together with  with $\lceil \frac{n}{x}\rceil =k+1\le y$ implies that 
$\lceil \frac{n}{y}\rceil \le \lceil \frac{n}{k+1}\rceil \le x$, because $x-m$ is positive. The reverse implication follows by symmetry. 
\end{proof}

\begin{lemma}[First double floor lemma]\label{fdf}
Let  $s$ and $ n$ be two natural numbers with $0 < s \le n$ and
\begin{equation}\label{fdfl} n=a \cdot s + b \text{ with } 0 \le b < s.\end{equation}
Then 
$\lfloor n/{\lfloor\frac{n}{s}\rfloor}\rfloor = s$ iff $b<a$; otherwise, $\lfloor n/{\lfloor\frac{n}{s}\rfloor}\rfloor > s$.
\end{lemma}

\begin{proof}
From~\eqref{fdfl} we have that $\fl{\frac{n}{s}}=a$. Therefore, 
$$\fl{ n/{\fl{\frac{n}{s}}}} = \fl{\frac{n}{a}}= \fl{ s+\frac{b}{a}}=  s+\fl{\frac{b}{a}}.$$The latter is equal to $s$ iff $b<a$, and otherwise, is bigger than $s$.
 \end{proof}

The next lemma deals with the case where $n$ is expressed as a multiple of a non-integer. 

\begin{lemma}[Second double floor lemma]\label{sdf}
Let  $0 \le m< n$ be two natural numbers and $f$ be a real number such that $0 \le f<1$. Define $\ell = \fl{\frac{n}{m+f}}$ and let $a$ and $b$ be the unique integers such that 
\begin{equation}\label{sdfl} n=a \cdot \ell + b \text{ with } 0 \le b < \ell.\end{equation}
If $m=\fl{\frac{n}{\ell}} $ then $\frac{b}{\ell}\ge f$.
\end{lemma}

\begin{proof}
By the definition of $\ell$, we have that $n=\ell \cdot (s+f) +c$, with $0 \le c<s+f$, where $c$ is not necessarily an integer. Then, $s=\fl{\frac{n}{\ell}}=\fl{s+f+\frac{c}{\ell}}=\fl{s+\frac{\ell \cdot f+c}{\ell}}=s+\fl{\frac{\ell \cdot f+c}{\ell}}$, and therefore, $\ell\cdot f+c<\ell$. Note that we can express $n$ as a multiple of $\ell$ as $n=s\cdot\ell+(\ell \cdot f+c)$, where $(\ell \cdot f+c)$ is an integer. Because $\ell\cdot f+c<\ell$, $\ell\cdot f+c$ is the residue modulo $\ell$ of $n$, that is $\ell \cdot f + c=b$ with $b$ as defined in \eqref{sdfl}. Thus, $\frac{b}{\ell}\ge f$.
 \end{proof}
 
 We are now ready for the alternative proof of Theorem~\ref{circuit length}.
 
 \begin{proof} 
{  We will make precise the immediate idea that if we want to distribute $\ell$ vertices as equidistant as possible among the $n$ vertices, then the distances should be roughly $n/\ell$.  Since we need integer values, the distances should be $\fl{n/\ell}$ and $\ceil{n/\ell}$. Let  $a$ and $b$ be integers such that
$$\label{floorcond} n=a \cdot \ell + b \text{ with } a=\fl{\frac{n}{\ell}} \text{ and } b=n\imod{\ell}.$$
By Theorem~\ref{circuit cond}, the conclusion will follow if we can distribute the  $\ell$ distances (and therefore determine the $\ell$ vertices) in such a way that
 \begin{enumerate} 
\item $d_i+d_{i+1} > s$   
and
\item $d_i \le s$
\end{enumerate}
for $ i=1,2,\ldots,\ell$, where $d_{\ell+1}= d_1$.}
For a value of $\ell$ that satisfies~\eqref{circlen}, we will construct an $\ell$-subset as follows:
Of the $\ell$ distances $d_1, d_2, \ldots, d_{\ell}$, we define $b$ distances to have value $a+1$ and the remaining $\ell-b$ distances to have value $a$. This assignment  satisfies the condition that the sum of the distances be $n$. By assumption, $n/s \le \ell$ and $\ell$  is an integer, so $\ceil{n/s}\le \ell$ and Lemma~\ref{rec} implies that $\max(d_i)=\ceil{n/\ell}\le s$, so  the second circuit condition holds for all values of $\ell$ satisfying ~\eqref{circlen}. 

{ To show the first circuit condition on the sums of consecutive distances, we consider the two cases $b=0$ and $b>0$ separately. If $b=0$, then $a=\frac{n}{\ell}$ and all vertices are distance $a$ apart. Since by assumption $\ell \le \frac{2n}{s+1}$,  $\frac{s+1}{2} \le \frac{n}{\ell} = a$, and therefore, $s+1 \le 2a = d_i+d_{i+1}$, so Theorem ~\ref{circuit length} follows in this case.

If $b>0$, we show that the first circuit condition holds when $\ell$ equals the upper bound $\fl{\frac{2n}{s+1}}$, and then show that the implication is also true for smaller values of $\ell$. Note that if there are at least as many distances of value $a+1$ as there are of value $a$ (that is, if $b \ge \ell - b$, or equivalently, $\ell/2 \le b$), then it is possible to order the distances in such a way that no two consecutive distances have value $a$, and $d_i + d_{i+1} \ge 2a+1$ for all $i$. Otherwise, there will be a pair of consecutive distances whose sum $d_i + d_{i+1} = 2a$ is minimal.}

Now let $\ell =\fl{\frac{2n}{s+1}}$ and assume that $s$ is odd, say $s=2m+1$. Then $\ell =\fl{\frac{n}{m+1}}$ and $\min(d_i)=\fl{n/\ell}=\fl{n/\fl{\frac{n}{m+1}}}\ge m+1$ (by Lemma~\ref{fdf}), so $d_i+d_{i+1}\ge 2(m+1)=s+1>s$. In the second case when  $s=2m$, then $\ell=\fl{\frac{n}{m+(1/2)}}\le \fl{\frac{n}{m}}$ and therefore, $\fl{\frac{n}{\ell}}\ge \fl{n/ \fl{\frac{n}{m}}}\ge m$, where the second inequality follows once more from Lemma~\ref{fdf}. If $\fl{\frac{n}{\ell}}\ge m+1$, then  $d_i+d_{i+1}>s$ as before. If $\fl{\frac{n}{\ell}}= m$, then Lemma~\ref{sdf} (with $f=1/2$) implies that $b/\ell \ge \frac{1}{2}$, that is, $b \ge \ell/2$. But this is precisely the case where in our construction $d_i+ d_{i+1} \ge 2a+1=2\fl{\frac{n}{\ell}}+1=2m+1=s+1>s$, and therefore the first circuit condition holds when $\ell$ equals the upper bound of~\eqref{circlen}.

{  It remains to be shown that if the first circuit condition holds for $\ell$, then it also holds for a value  $\ell'\le \ell$.  Let  $d'_1, d'_2, \ldots, d'_{\ell}$ be the distances that we obtain with $\ellÕ$ in our construction, and let $n=a' \cdot \ell'+b'$. Then
$a'=\fl{\frac{n}{\ell'}}\ge\fl{\frac{n}{\ell}}=a$, and therefore the only case we need to consider is the case for $s=2m$ and $a'=a$. Since 
$$n=a' \cdot \ell'+b'=a\cdot(\ell+(\ell'-\ell))+b'=a\cdot  \ell+(b'-a(\ell-\ell'))=a\cdot\ell+b,$$
we have that $b'\ge b\ge \ell/2 \ge \ell'/2$, and therefore we can distribute the distances $a'$ and $a'+1$ such that $d'_i+d'_{i+1}=2a'+1=2a+1\ge 2m+1>s$, and the proof is  complete.}
 \end{proof}

\section{Open questions}
We suggest a number of open questions for the interested reader to try his or her hand. 
\begin{enumerate}
\item General results for $\L_{\c(n,k)}$ for specific values of $k$:  We have derived general results for $\c(n,1)$, $\c(n,n-1)$, and $\c(n,n)$, the cases where $k$ is either small or large.  We also investigated general results for intermediate values of $k$, specifically $\c(2m,m)$. Recall that 
$$\L_{\c(4,2)}=\{(a,b,c,d)\mid a+b=c+d \wedge b+c=a+d\}\hspace{0.1in}\text{and}$$
$$\L_{\c(6,3)}=\{(a,b,c,d,e,f)| a+b = d+e \wedge b+c = e+f\}.$$
A natural conjecture for $\c(2m,m)$ based on $m=2,3$ could be that sums of pairs that are diagonally across from each other are the same, as indicated in Figure~\ref{diag conj}. Unfortunately, we found a counterexample for this conjecture. 

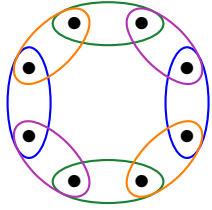
\begin{figure}[htp]
 \begin{center}
 { \psset{xunit=0.3,yunit=0.3}
\begin{pspicture}(11,11)
 \rput(4,2){$\bullet$}
 \rput(7,2){$\bullet$}
  \rput(4,9){$\bullet$}
 \rput(7,9){$\bullet$}
  \rput(2,4){$\bullet$}
 \rput(2,7){$\bullet$}
  \rput(9,4){$\bullet$}
 \rput(9,7){$\bullet$}
\psellipse[linecolor=mygreen](5.5,2)(2.5,1)
\psellipse[linecolor=mygreen](5.5,9)(2.5,1)
\psellipse[linecolor=blue](2,5.5)(1,2.5)
\psellipse[linecolor=blue](9,5.5)(1,2.5)
\psccurve[linecolor=purple](9.5,6.5)(8.7,8.7)(6.5,9.5)(7.3,7.3)
\psccurve[linecolor=purple](4.5,1.5)(3.7,3.7)(1.5,4.5)(2.3,2.3)
\psccurve[linecolor=orange](1.5,6.5)(2.3,8.7)(4.5,9.5)(3.7,7.3)
\psccurve[linecolor=orange](6.5,1.5)(7.3,3.7)(9.5,4.5)(8.7,2.3)
\end{pspicture}}
\caption{Conjecture for the losing set of $\c(2m,m)$}\label{diag conj}
\end{center}
\end{figure}

\item Results for  $n=7$: We know next to nothing about $\c(7,k)$. Recursively computed losing positions  for $k=3, 4, 5$ seem to always have an empty stack.
\item Use of subgame structure: It is easy to see that $\c(n,k)$ contains $\c(3,1)$ for $n \ge 3k$ (when $k-1$ empty stacks are followed by a non-empty stack). Can one make use of this fact (and similar subgames)? At minimum this fact indicates that the losing sets for larger values of $n$ will not be closed under addition. 
\item Finally, there are numerous variations on this game. Here are a few:
\begin{itemize}
\item  Select a  fixed number $a$ from at least one of the stacks;
\item Select a fixed number $a$ from each of the heaps;
\item  Select at least $a$ tokens from each of the $k$ heaps;
\item Select a  total of at least $a$ tokens from the $k$ stacks;
\item  Select a total of exactly $a$ tokens from the $k$ stacks.
\end{itemize}
\end{enumerate}


\bibliographystyle{plain}\pagestyle{headings}
\bibliography{CombGame}

\end{document}